\documentclass[12pt,final,2p,times,a4paper]{elsarticle}
\usepackage[a4paper, margin=1in]{geometry}  % Un-comment and use the geometry package for margin adjustment
\usepackage{newtxtext}
\usepackage{bm}
\usepackage{amssymb}
\usepackage{amsmath}
\usepackage{amsthm}
\usepackage{mathtools}
\usepackage{graphicx}
\usepackage{color} 
\usepackage{multicol}
\usepackage{subfig}
\usepackage{varwidth}
\usepackage{float}
\usepackage{booktabs}  % for tables with multiple columns
\usepackage{array}     % for tables with multiple columns
\usepackage{multirow}
\usepackage{natbib}
\usepackage{bm}
 \usepackage{url}
\usepackage{mathtools}
\usepackage{xcolor}
\pagecolor{white}  % Ensure pure white background
\color{black}  % Ensure black text
\newtheorem{theorem}{Theorem}
\newtheorem{lemma}{Lemma}

\newtheorem{definition}{Definition}
\newtheorem{proposition}{Proposition}
\newtheorem{corollary}{Corollary}
\newtheorem{example}{Example}
\newtheorem{remark}{Remark} 
\newcommand{\bea}{\begin{eqnarray}}
\newcommand{\eea}{\end{eqnarray}}

\allowdisplaybreaks 
\usepackage{etoolbox}  % Load etoolbox for patching environments
\usepackage{cleveref}
\usepackage{cleveref}
\crefname{theorem}{Theorem}{Theorems}
\crefname{proposition}{Proposition}{Propositions}
\crefname{lemma}{Lemma}{Lemmas}
\crefname{corollary}{Corollary}{Corollaries}
\crefname{definition}{Definition}{Definitions}
\crefname{remark}{Remark}{Remarks}
\crefname{example}{Example}{Examples}
\crefname{table}{Table}{Tables}
\AtBeginEnvironment{proof}{\normalfont}  % Ensures proof text is upright
\makeatletter
\renewenvironment{proof}[1][\proofname]{%
  \par\pushQED{\qed}%
  \normalfont\topsep6pt \partopsep0pt % spacing
  \trivlist
  \item[\hskip\labelsep
        \bfseries #1.]\ignorespaces % makes "Proof." bold
}{%
  \popQED\endtrivlist\@endpefalse
}
\makeatother
\makeatletter
\def\ps@pprintTitle{%
 \let\@oddhead\@empty
 \let\@evenhead\@empty
 \let\@oddfoot\@empty
 \let\@evenfoot\@empty
}
\makeatother
\begin{document} 
\begin{frontmatter}
\title{Weighted Tail Random Variable: A Novel Framework with Stochastic Properties and Applications}
%\title{On the Study of the Weighted Tail Random Variable: A Novel Framework with Stochastic Properties and Applications}
%\title{On the Stochastic Properties and Applications of the Weighted Tail Random Variable}
%\title{On the Introduction of the Weighted Tail Random Variable: Stochastic Properties and Applications}
\author{Sarikul Islam $^{*}$\quad and\quad Nitin Gupta $^{**}$ }
\address{\normalfont $^{*,\,**}$ Department of Mathematics, 
Indian Institute of Technology Kharagpur, 721302, West Bengal, India.\\
Email address: $^{*}$sarikul\_phd\_math@kgpian.iitkgp.ac.in, \: $^{**}$nitin.gupta@maths.iitkgp.ac.in.}
\begin{abstract}
This paper introduces a novel framework to construct the probability density function (PDF) of non-negative continuous random variables. The proposed framework uses two functions: one is the survival function (SF) of a non-negative continuous random variable, and the other is a weight function, which is an increasing and differentiable function satisfying some properties. The resulting random variable is referred to as the weighted tail random variable (WTRV) corresponding to the given random variable and the weight function. We investigate several reliability properties of the WTRV and establish various stochastic orderings between a random variable and its WTRV, as well as between two WTRVs. Using this framework, we construct a WTRV of the Kumaraswamy distribution. We conduct goodness-of-fit tests for two real-world datasets, applied to the Kumaraswamy distribution and its corresponding WTRV. The test results indicate that the WTRV offers a superior fit compared to the Kumaraswamy distribution, which demonstrates the utility of the proposed framework.
\end{abstract}
\end{frontmatter} 
\section{Introduction} 
Probability models are fundamental tools for representing and analyzing a wide range of real-world data, including time-to-event data in survival analysis and failure-time data in reliability engineering.  Many probability distributions, including the beta, exponential, gamma, Kumaraswamy, logistic, and Weibull models, may sometimes fail to capture some features of certain real-world data. Their functional shapes may not consistently provide an optimal fit for specific data, indicating the necessity for the development of more adaptable models.  Researchers have proposed various generalizations of classical probability distributions to enhance flexibility and expand applicability across multiple disciplines. Many of these methods may achieve greater flexibility by introducing one or more new parameters to the baseline distribution, employing transformation techniques, or combining existing distributions. \par
\citet{barreto2010beta} introduced the beta generalized exponential model, which includes the beta exponential and generalized exponential models as its special cases. They presented a probability model for positive data and thoroughly investigated its properties and parameter estimation methods. \citet{ALEXANDER20121880} presented the generalized beta-generated (GBG) distributions, which include classical beta-generated, Kumaraswamy-generated, and exponentiated distributions as specific cases. It emphasized the flexibility attributes and improved modeling capabilities by providing an additional shape parameter, increasing central entropy while preserving tail dynamics. \citet{Cordeiro01072010} introduced a family of generalized probability models, which includes conventional models such as the gamma, normal, Weibull, and Gumbel distributions. They expressed the ordinary moments of the generalized distribution as linear functions of the probability weighted moments (PWMs) of the parent distribution, and illustrated the versatility of the new family through applications using real data. \citet{Eugene14052002} presented a new class of probability models derived from the logit of the beta distribution, highlighting the beta-normal distribution as a particular case. Their research demonstrated flexibility in modeling symmetric, skewed, and bimodal data, supported by maximum likelihood estimates and empirical analysis. Some researchers have proposed probability frameworks for constructing new probability distributions by leveraging one or more existing ones as foundational components. \citet{alzaatreh2013new} presented a novel framework for generating continuous distribution families using a transformation framework. The generated family is called the T--X family of probability models. Several well-known probability models are particular cases of the T--X family. For a comprehensive account of the generalized probability models and their applications, see \citet{mandouh2024new}, \citet{SAGRILLO2021127021}, \citet{Tahir2016CompoundingOD}, and the references therein.  To the best of our knowledge, no established probabilistic framework has been proposed in the literature that systematically harnesses the SF of time-to-event distributions as a foundational tool for innovating new classes of non-negative continuous distributions. \par
 In this paper, we adopt a novel approach to develop a framework for constructing non-negative continuous probability models. The proposed framework utilizes the SF of a non-negative continuous probability model and incorporates an additional function, referred to as a weight function, to formulate the PDF of another random variable. Such a framework can preserve some characteristics of the input random variable while offering enhanced flexibility and superior goodness-of-fit across distributions with differing asymmetry and tail behavior, compared to the input random variable of the framework. \par
Consider a random variable \(X\) with lower bound \(0\), upper bound \(u>0\), distribution function \(F_X(\cdot)\), and SF \(\overline{F}_X(\cdot) = 1 - F_X(\cdot)\). Further, consider a non-constant, increasing and differentiable function \(w: [0, u) \to [0, \infty)\) satisfying \(w(0) = 0\) and
\[
\int_{0}^{u} \int_{0}^{x} \left|w^{\prime}(t)\right| \, dt \, dF_X(x) < \infty.
\]
Then we define the random variable \(X_w\), referred to as the WTRV of \(X\) with weight function \(w(\cdot)\), by the PDF \(f_{X_w}(\cdot)\) as follows:
\begin{equation}
f_{X_w}(x) = \frac{w^{\prime}(x) \, \overline{F}_X(x)}{\mathbb{E}[w(X)]}, \quad 0\leq x <u, \label{1} 
\end{equation} 
where \(\mathbb{E}[w(X)]\) denotes the expectation of \(w(X)\), and if \(X\) is unbounded above, then \(u=\infty\). This framework enables the systematic construction of several well-known distributions and, in some cases, introduces new shape parameters. Examples include the WTRV of beta, Burr Type XII, gamma, Kumaraswamy, logistic, and Weibull distributions, each obtained by selecting an appropriate non-negative continuous random variable and suitable weight function (see \cref{table1}).\par
This framework encompasses an important class of random variables in renewal theory, namely, equilibrium random variables. The duration between a specified time \(t\) and the following renewal in a renewal process is known as the forward recurrence time, and the equilibrium distribution describes the limiting distribution of the forward recurrence time of the renewal process, also known as the asymptotic equilibrium distribution. Numerous studies have explored equilibrium random variables and their reliability properties in recent decades. Let \( X \) be a non-negative continuous random variable with a finite mean \( \mu > 0 \). If we set \( w(x) = x,\: x\geq 0\) in the proposed framework given by Equation~\eqref{1}, then the PDF of \(X_w\) becomes
 \begin{equation} 
 f_{X_w}(x)=\frac{\overline{F}_X(x)}{\mu}, \quad x\geq 0.\label{2}
 \end{equation} 
This PDF characterizes the equilibrium distribution of the variable \(X\); see \citet{bon2005ageing}. We denote the corresponding random variable by \(\widetilde{X}\). Hence, if \( w(x) = x,\: x\geq 0\), then the framework constructs the equilibrium random variable \(\widetilde{X}\). One may refer to \citet{bon2005ageing}, \citet{gupta2007role}, \citet{gupta1998bivariate}, \citet{li2008reversed}, \citet{Nanda1996sorder}, and the references therein for relevant literature on the equilibrium random variable. \par
We investigate various reliability properties of \(X_w\) by considering some reliability properties of \(X\) and some conditions on the weight function $w(\cdot)$. We also investigate the conditions on  \(X\) and \(w(\cdot)\) under which the aging properties of \(X\), such as decreasing failure rate (DFR), increasing failure rate (IFR), decreasing mean residual life (DMRL), and increasing mean residual life (IMRL) are preserved by \(X_w\). Subsequently, we assume that the variables \(X\) and \(Y\) are stochastically ordered and the function \(w(\cdot)\) satisfies certain properties; we then establish various stochastic ordering relationships between \(X_w\) and \(Y_w\). Furthermore, we investigate conditions on the stochastically ordered variables \( X \) and \( Y \), as well as on \( w(\cdot) \), under which \( X_w \) and \( Y_w \) preserve the presumed stochastic order between \( X \) and \( Y \). We specifically examine the preservation of presumed stochastic orderings between \(X\) and \(Y\), such as failure rate ordering, reversed failure rate ordering, and the usual stochastic ordering by the corresponding WTRVs \(X_w\) and \(Y_w\). Some of the results in this paper generalize existing results on equilibrium random variables. \par
By appropriately selecting the weight function \( w(\cdot)\), we may construct a variable \( X_w \) that provides superior fits to empirical data in contrast to the variable \( X\). As an application of the proposed framework, we constructed a novel distribution, termed the weighted Kumaraswamy distribution, by selecting the weight function \( w(x) = x^c \), \(0\leq x <1\), where \( c > 0 \), and assuming that \(X\) follows the Kumaraswamy distribution.  The resulting model, with three shape parameters, offers greater flexibility for modeling empirical data. Goodness-of-fit tests on monsoon rainfall data from India (1901–2021) suggest that the weighted Kumaraswamy distribution outperforms the Kumaraswamy distribution for modeling the mentioned rainfall data. This demonstrates the utility of the framework in constructing flexible distributions that can capture the characteristics of real-world data with greater accuracy. \par
The structure of this paper is as follows. Section~2 provides the notations, preliminary concepts, and lemmas useful for subsequent sections. Section~3 provides the methodology of the proposed framework and presents the WTRVs of some conventional distributions in \cref{table1}. Section~4 is dedicated to the analysis of various reliability properties of a random variable and its WTRV. Section~5 establishes various stochastic orders involving random variables and their WTRV. It also investigates conditions under which several stochastic orders between two random variables are preserved by their corresponding WTRVs. Section~6 provides an application of the proposed framework to demonstrate its usefulness in real-world data modeling. Section~7 concludes the paper.
\section{Notations, Preliminary Concepts, and Lemmas}
This section presents the notations, key concepts, and lemmas that are useful for the later sections of this paper. \par
The term ‘increasing (decreasing) function’ refers to a monotonically increasing (monotonically decreasing) function, unless stated otherwise. The PDF of a random variable \(X\) is denoted by \(f_X(\cdot)\); the cumulative distribution function (CDF) by \(F_X(\cdot)\); the SF by \(\overline{F}_X(\cdot)=1-F_X(\cdot)\); the failure rate function by \(r_X(\cdot)=f_X(\cdot)/\overline{F}_X(\cdot)\); the reversed failure rate function by \(\widetilde{r}_X(\cdot)=f_X(\cdot)/F_X(\cdot)\); the mean residual life (MRL) function by \(m_X(\cdot)=\int_{\cdot}^{\infty} \overline{F}_X(x) \, dx/\overline{F}_X(\cdot)\); the support of the PDF \(f_X(\cdot)\) of the variable \(X\) by \(S_X = (l_X,\, u_X)\), where \(l_X=\inf\{x\in \mathbb{R}\:: f_X(x)>0\}\) and \(u_X=\sup\{x\in\mathbb{R} \:: f_X(x)>0\}\), and we say that \(X\) has the support \(S_X\). The random variable \(X_{w}\), with its support \(S_1=(l_1,\,u_1)\) such that \(l_1=\inf\{x\in \mathbb{R}\: : f_{X_{w}}(x)>0\}\) and \(u_1=\sup\{x\in\mathbb{R} \:: f_{X_{w}}(x)>0\}\), denotes the WTRV of \(X\) with weight function \(w(\cdot)\). The variable \(Y_{w}\), with its support \(S_2=(l_2,\,u_2)\) such that \(l_2=\inf\{x\in \mathbb{R}\: : f_{Y_{w}}(x)>0\}\) and \(u_2=\sup\{x\in\mathbb{R} \:: f_{Y_{w}}(x)>0\}\), denotes the WTRV of \(Y\) with weight function \(w(\cdot)\).  From now on, \( X \) and \( Y \) denote non-negative, continuous random variables with respective supports \( S_X = (l_X,\,u_X) \) and \( S_Y = (l_Y,\,u_Y) \), unless specified otherwise. The functions \( w(\cdot) \), \( w_1(\cdot) \), and \( w_2(\cdot) \) denote weight functions, while their first derivatives are denoted by \( w'(\cdot) \), \( w_1'(\cdot) \), and \( w_2'(\cdot) \), respectively.\par
We now present several definitions that will be used in the subsequent sections of the paper.
\begin{definition}\label{definition1}(Definition 2 of \citet{LIU2020108863}) Consider the random variable \( X \) with its lower bound \( l_X \in \mathbb{R} \cup \{-\infty\} \) and upper bound \( u_X \in \mathbb{R} \cup \{\infty\} \). A non-constant function \( w(\cdot) \) is called absolutely integrable with respect to \( X \) if it fulfills one of the following conditions:
\begin{enumerate}
    \item[(i)] \( l_X \) is a finite lower bound of \( X \), and
    \[
    \int_{l_X}^{u_X} \int_{l_X}^{x} |w(t)| dt \, dF_X(x) < \infty.
    \]
    Let \( u_X \to \infty \) when \( X \)  is unbounded above.
    \item[(ii)] \( u_X \) is a finite upper bound of \( X \), and
    \[
    \int_{l_X}^{u_X} \int_{x}^{u_X} |w(t)| dt \, dF_X(x) < \infty.
    \]
    Let \( l_X \to -\infty \) when \( X \) is unbounded below.
\end{enumerate}
\end{definition}\par
We now present the definitions of several stochastic orderings between two random variables: (see \citet{Misra31012008})
\begin{definition}\label{definition2} Let \(S_X=(l_X,\,u_X)\), and \(S_Y=(l_Y,\,u_Y)\) be the respective supports of \(X\) and \(Y\). Then \(X\) is said to be less than \(Y\) in the
\begin{itemize}
    \item[(i)] likelihood ratio (lr) ordering (expressed as \( X \leq_{lr} Y \)), if \(\: f_Y(x_1)f_X(x_2) \leq f_X(x_1)f_Y(x_2)\) for all \(x_1,\,x_2\in (-\infty,\, \infty)\) such that \(x_1 \leq x_2\); or equivalently, if \(l_X \leq l_Y\), \(u_X \leq u_Y\), and \(f_Y(\cdot)/f_X(\cdot)\) is increasing on \(S_X \cap S_Y\);
    
    \item[(ii)] failure rate (fr) ordering (expressed as \( X \leq_{fr} Y \)), if \(\: \overline{F}_Y(x_1)\overline{F}_X(x_2) \leq \overline{F}_X(x_1)\overline{F}_Y(x_2)\) for all \(x_1,\,x_2\in (-\infty,\, \infty)\) such that \(x_1 \leq x_2\); or equivalently, if \( l_X \leq l_Y \), \( u_X \leq u_Y \), and \( r_X(x) \geq r_Y(x) \) for all \( x \in (-\infty, u_X) \); or equivalently, if \(\overline{F}_Y(\cdot)/\overline{F}_X(\cdot)\) is increasing on \((-\infty,\, \max(u_X,\,u_Y))\);
    
    \item[(iii)] reversed failure rate (rfr) ordering (expressed as \(X\leq_{rfr} Y \)), if \(\: F_Y(x_1)F_X(x_2) \leq F_X(x_1)F_Y(x_2)\) for all \(x_1,\,x_2\in (-\infty,\, \infty)\) such that \(x_1 \leq x_2\); or equivalently, if \( l_X \leq l_Y \), \( u_X \leq u_Y \), and \( \widetilde{r}_X(x) \leq \widetilde{r}_Y(x) \) for all \( x \in (l_Y, \infty) \); or equivalently, if \( F_{Y}(\cdot)/F_{X}(\cdot) \) is increasing on \((\min(l_X,l_Y),\,\infty)\);

    \item[(iv)] usual stochastic (st) ordering (expressed as \(X \leq_{st} Y\)), if \(\:\overline{F}_X(x) \leq \overline{F}_Y(x)\) for every \(x\in(-\infty,\, \infty)\); or equivalently, if \(l_X \leq l_Y\), \(u_X \leq u_Y\), and \( \overline{F}_X(x) \leq \overline{F}_Y(x)\) for all \(x \in S_X \cap S_Y\).
\end{itemize}
\end{definition}
We proceed by defining a log-concave [log-convex] function from \citet{Misra31012008}.
\begin{definition} \label{definition3} Let \( f : (a, b) \rightarrow (0, \infty) \) be a function such that \( \log f(\cdot) \) is bounded below [bounded above] in a neighborhood of a point \( x_0 \in (a, b) \). Then \( f \) is said to be \emph{log-concave} [\emph{log-convex}] on \( (a, b) \) if, for all \( 0 < \theta < b - a \), \(x_1\) and \(x_2\) such that \(a < x_1 < x_2 < b - \theta\), the following inequality holds:
\[
f(x_1 + \theta) f(x_2) \geq [\leq]\, f(x_1) f(x_2 + \theta).
\]\end{definition} \par
We now present definitions of some reliability properties of a random variable useful in the subsequent sections of this paper: (see \citet{Misra31012008}). 
\begin{definition}\label{definition4} Consider the variable \( X \) with the support \(S_X=(l_X,\,u_X)\). Then, \( X \) is said to be an
\begin{itemize}
  \item[(i)] increasing likelihood ratio (ILR) [decreasing likelihood ratio (DLR)] if the PDF \( f_X(\cdot) \) is log-concave [log-convex] on \( S_X \), or equivalently, if the Glaser function \(\eta_X(\cdot)=-f_X'(\cdot)/f_X(\cdot)\) is increasing [decreasing] on \(S_X\);
  
  \item[(ii)] IFR [DFR] if the SF \( \overline{F}_X(\cdot) \) is log-concave [log-convex] on \( S_X \), or equivalently, if \(r_X(\cdot)\) is increasing [decreasing] on \(S_X\);
  \item[(iii)] IMRL [DMRL] if \(m_X(\cdot)\) is increasing [decreasing] on \(S_X\). 
\end{itemize}
\end{definition}\par
One may refer to \citet{shaked2007stochastic} and the references therein for details of stochastic orderings and stochastic aging/reliability notions.\par

The following facts on aging properties of a random variable are useful: (see \citet{Misra31012008}).
\begin{enumerate}
    \item[(i)] If $f_X(\cdot)$ is log-concave, then $X$ has the IFR property;
    \item[(ii)] If $X$ has IFR property, then it has DMRL property;
    \item[(iii)] If $X$ has DFR property, then it has IMRL property, which further implies $u_X = \infty$, and hence \(S_X=(l_X, \infty)\);
    \item[(iv)] If $f_X(\cdot)$ is log-convex and $u_X = \infty$, then $X$ has DFR property.
\end{enumerate} \par
We now present two lemmas from \citet{Misra31012008}, which are used in subsequent sections to establish several theorems.
\begin{lemma}\label{lemma1} Consider the independent variables \( X \) and \( Y \) having their respective supports \(S_X=(l_X,\,u_X)\) and \(S_Y=(l_Y,\,u_Y)\). Then \( X \leq_{st} Y \) if and only if \( l_X \leq l_Y \), \( u_X \leq u_Y \), and \( \mathbb{E}[\phi_2(X,Y)] \geq \mathbb{E}[\phi_1(X,Y)] \) for all such functions \( \phi_2(\cdot,\cdot) \), \( \phi_1(\cdot,\cdot) \), and \( \Delta\phi_{21}(\cdot,\cdot) = \phi_2(\cdot,\cdot) - \phi_1(\cdot,\cdot) \) that satisfy conditions (i)-(iv) outlined below:
\begin{itemize}
 \item[(i)] \( \Delta\phi_{21}(x,y) \geq 0 \), whenever \( (x,y) \in \overline{S}_X \times \overline{S}_Y \) and \( x \leq y \);
  \item[(ii)] If \( S_X \cap S_Y \neq \emptyset \), then for every fixed \( x \in \overline{S}_X \), the function \( \Delta\phi_{21}(x,y) \) is non-decreasing in \( y \in [\max(x, l_Y), u_Y) \);
  \item[(iii)] If \( S_X \cap S_Y \neq \emptyset \), then for every fixed \( y \in \overline{S}_Y \), the function \( \Delta\phi_{21}(x,y) \) is non-increasing in \( x \in [l_X, \min(y, u_X)) \);
  \item[(iv)] If \( S_X \cap S_Y \neq \emptyset \), then \( \Delta\phi_{21}(x,y) +\Delta\phi_{21}(y,x)\geq 0 \) for all \( (x,y) \in (\overline{S}_X \cap \overline{S}_Y) \times (\overline{S}_X \cap \overline{S}_Y) \) such that \( x \leq y \);
\end{itemize}
where \( \overline{S}_X = [l_X, u_X] \) and \( \overline{S}_Y = [l_Y, u_Y] \) denote the closure of \(S_X\) and \(S_Y\), respectively.
\end{lemma}
\begin{lemma}\label{lemma2} Consider the independent variables \( X \) and \( Y \) having their respective supports \(S_X=(l_X,\,u_X)\) and \(S_Y=(l_Y,\,u_Y)\). Then \( X \leq_{fr} Y \) if and only if \( l_X \leq l_Y \), \( u_X \leq u_Y \), and \(\mathbb{E}[\phi_2(X, Y)] \geq \mathbb{E}[\phi_1(X, Y)]\)
for all functions \( \phi_2(\cdot, \cdot) \), \( \phi_1(\cdot, \cdot) \), and $\Delta \phi_{21}(\cdot, \cdot) = \phi_2(\cdot, \cdot) - \phi_1(\cdot, \cdot)$ that satisfy conditions (i)-(iii) outlined below:
\begin{itemize}
 \item[(i)] \( \Delta \phi_{21}(x, y) \geq 0 \), for all \( (x, y) \in S_X \times S_Y \) such that \( x \leq y \);
  \item[(ii)] If \( S_X \cap S_Y \neq \emptyset \), then for every fixed \( x \in S_X \cap S_Y \), the function \( \Delta \phi_{21}(x, y) \) is increasing in \( y \in [x, u_Y) \);
  \item[(iii)] If \( S_X \cap S_Y \neq \emptyset \), then for all \( (x, y) \in (S_X \cap S_Y) \times (S_X \cap S_Y) \) such that \( x \leq y \),  The following inequality holds: \[ \Delta \phi_{21}(x, y)  +\Delta \phi_{21}(y, x)\geq 0.\]
\end{itemize}
\end{lemma}\par
The following lemma from \citet[p.~20]{ThesisNGupta2009} will be used later to establish a theorem concerning the reversed failure rate order.
\begin{lemma}\label{lemma3} Consider the independent variables \( X \) and \( Y \) having their respective supports \(S_X=(l_X,\,u_X)\) and \(S_Y=(l_Y,\,u_Y)\). Then \( X \leq_{rfr} Y \) if and only if \( l_X \leq l_Y \), \( u_X \leq u_Y \), and \(\mathbb{E}[\phi_2(X, Y)] \geq \mathbb{E}[\phi_1(X, Y)]\)
for all functions \( \phi_2(\cdot, \cdot) \), \( \phi_1(\cdot, \cdot) \), and $\Delta \phi_{21}(\cdot, \cdot) = \phi_2(\cdot, \cdot) - \phi_1(\cdot, \cdot)$ satisfying conditions (i)-(iii) outlined below:
\begin{itemize}
 \item[(i)] \( \Delta \phi_{21}(x, y) \geq 0 \), for all \( (x, y) \in S_X \times S_Y \) such that \( x \leq y \);
  \item[(ii)] If \( S_X \cap S_Y \neq \emptyset \), then for every fixed \( y \in S_X \cap S_Y \), the function \( \Delta \phi_{21}(x, y) \) is decreasing in \( x \in (l_X, y)\);
  \item[(iii)] If \( S_X \cap S_Y \neq \emptyset \), then for all \( (x, y) \in (S_X \cap S_Y) \times (S_X \cap S_Y) \) such that \( x \leq y \), The following inequality holds: \[ \Delta \phi_{21}(x, y)  +\Delta \phi_{21}(y, x)\geq 0.\]
\end{itemize}
\end{lemma}
\section{Methodology of the Framework} 
This framework constructs the PDF of a new random variable utilizing:  
\begin{itemize} 
    \item[(i)] the SF \( \overline{F}_X(\cdot) \) of a non-negative continuous random variable \( X \), and;
    \item[(ii)] a weight function \(w(\cdot)\) that is increasing and differentiable on \((0,\, u_X)\) satisfies \(w(0)=0\), and has its first derivative absolutely integrable with respect to \(X\), where \(u_X\) is the upper bound of \(X\).
\end{itemize}
The foundation of this framework is built on well-established theoretical results, which are presented chronologically to support its development. The formula for the expectation of a non-negative variable \( X \) is given as  
 \begin{equation*}
    \mathbb{E}[X] = \int_{0}^{\infty} \overline{F}_X(x) \,dx.
\end{equation*}
This formula is often called an ‘alternative formula’ for computing the expectation of  \(X\). A generalization of the above result given by \citet{hong2012remark} and states that if  \(X\) has finite \(k\)-th moment, that is, \(\mathbb{E}|X^k| < \infty\) for any integer \( k > 0 \), then \( \mathbb{E}[X^k] \) is given by:  
\begin{equation}
    \mathbb{E}[X^k] = \int_{0}^{\infty} kx^{k-1} \overline{F}_X(x) \,dx. \label{3}
\end{equation}
 \citet{ogasawara2020alternative} provides an alternative formula for the expectation of a strictly increasing function of a continuous variable. \par
A further generalization of the alternative formula for computing the expectation of a function of \(X\) is established in \citet{LIU2020108863}. The formula is given by the following lemma: (see Theorem 1 of \citet{LIU2020108863}). 
\begin{lemma}\label{lemma4} Let the continuous variable \( X \) has support \( (l_X, u_X) \subset \mathbb{R} \). Let a differentiable function \( w(\cdot) \) be defined on \( (l_X, u_X) \) with its first derivative absolutely integrable with respect to \(X\).
\item If \( l_X \) is real-valued, then
\begin{align}
\mathbb{E}[w(X)] = w(l_X) + \int_{l_X}^{u_X} w'(x) \overline{F}_X(x) \, dx. \label{4}
 \end{align}
Let \( u_X \to \infty \) when \( X \)  is unbounded above.
\end{lemma}
This lemma is crucial for the proposed framework. \cref{lemma4} indicates that the expectation of \(w(X)\) can be formulated using the derivative of \(w(\cdot)\) and the SF of \(X\), provided that \(w(\cdot)\) satisfies the conditions of \cref{lemma4}. This lemma is proved in \citet{LIU2020108863}.
\subsection{\textbf{Derivation of the Framework}}
By setting \( l_X = 0 \), \( w(0) = 0 \) in \cref{lemma4}, Equation~\eqref{4} becomes:
\begin{equation}
    \mathbb{E}[w(X)] = \int_{0}^{u_X} w'(x) \overline{F}_X(x) \,dx. \label{5}
\end{equation}
We further assume that the function \(w(\cdot)\), referred to as a weight function, is increasing; this condition is besides the assumptions on \(w(\cdot)\) mentioned in \cref{lemma4}. Then, clearly, the expectation \( \mathbb{E}[w(X)] \) is positive.
 We divide both sides of Equation~\eqref{5} by \( \mathbb{E}[w(X)] \), resulting in:
\begin{equation}
    \int_{0}^{u_X} \frac{w^{\prime}(x)\overline{F}_X(x)}{\mathbb{E}[w(X)]} \,dx = 1.\label{6}
\end{equation}
We denote the integrand of Equation~\eqref{6} by \(f_{X_w}(x)\). The function \( f_{X_w}(\cdot) \) satisfies the following:
\begin{itemize}
    \item[(i)] both \( w'(x) \) and \( \overline{F}_X(x) \) are non-negative, so \( f_{X_w}(x) \geq 0 \) for all \( x \in (0 \,,\,u_X)\);
    \item[(ii)] by construction, \[ \int_{0}^{u_X} f_{X_w}(x) \, dx = 1\,. \]
\end{itemize}
Hence, \( f_{X_w}(\cdot) \) defines a PDF of a non-negative continuous probability model; let us denote the corresponding variable by \( X_w \) and refer to it as the WTRV of \( X \) with weight function \( w(\cdot) \). This methodology provides a systematic approach to construct a new probability model using the SF of a variable \(X\) and a weight function \(w(\cdot)\) defined on \([0, \,u_X)\). \par
\cref{table1} applies the proposed framework to construct several well-known probability models.
\begin{table}[H]
\caption{\large Construction of some distributions employing the proposed framework}
\label{table1} 
\centering
\renewcommand{\arraystretch}{2.6} % Adjust row height for readability
{\footnotesize \begin{tabular}{cccccc}  
    \toprule
    \(X\sim\) & $\overline{F}_X(x),\,x\geq0$ & \(w(x),\,x\geq 0\) &  \(f_{X_w}(x)=\frac{w'(x)\overline{F}_X(x)}{\mathbb{E}[w(X)]}\) & \(X_w\sim\) \\  
    \midrule 
    Exponential($\lambda$) & \(e^{-\lambda x},\,\lambda>0 \) & \(w_1(x)=x\) & \(\lambda e^{-\lambda x}\) & Exponential($\lambda$)  \\
    Exponential($\lambda$) &  \(e^{-\lambda x},\,\lambda>0\) & \(w_2(x)=x^k,\,k>0\) & \(\frac{\lambda^k x^{k-1}e^{-\lambda x}}{\Gamma(k)}\) & Gamma$(k, \lambda)$ \\ 
    Weibull$(\alpha,\beta)$ &\(e^{-(x/\beta)^\alpha},\,\alpha,\beta>0\) & \( w_3(x)=(x/\beta)^\alpha\) & \(\frac{\alpha}{\beta} \left(\frac{x}{\beta}\right)^{\alpha-1} e^{-\left(\frac{x}{\beta}\right)^\alpha}\)& Weibull$(\alpha,\beta)$ \\  
     $F(x, \beta),\,x\in (0,\,1)$ & $(1-x)^{\beta-1},\:\beta>1$ & \(w_4(x)=x^\alpha,\,\alpha>0\) & \( \frac{x^{\alpha - 1} (1 - x)^{\beta - 1}}{B(\alpha,\,\beta)}\) & Beta$(\alpha, \beta)$ \\  
      Exponential($\frac{1}{2}$) &  \(e^{-\frac{x}{2}}\) & \(w_5(x)=x^{\frac{k}{2}},\,k>0\) & \(\frac{(\frac{1}{2})^k x^{\frac{k}{2}-1}e^{-\frac{x}{2}}}{\Gamma(\frac{k}{2})}\) & Chi-Square$(k)$ \\ 
      Rayleigh\((\sigma)\) &\(e^{-\frac{x^2}{2\sigma^2}},\,\sigma>0\) & \( w_6(x)=\frac{x^2}{2\sigma^2}\) & \(\frac{x}{\sigma^2} e^{-x^2/2\sigma^2}\)& Rayleigh$(\sigma)$ \\ 
      Weibull\((2, \sqrt{2} \sigma)\)& \(e^{-x^2/2\sigma^2},\,\sigma>0\)& \(w_7(x)=\frac{x}{\sqrt{2} \sigma}\)&\(\frac{\sqrt{2}}{\sigma\sqrt{\pi}} e^{-x^2/2\sigma^2}\)& Half-Normal$(\sigma)$ \\
    Weibull$(p,a)$ &\(e^{-(x/a)^p},\,a,\,p>0\) & \( w_8(x)=(x/a)^p\) & \(\frac{px^{d-1} e^{-\left(\frac{x}{a}\right)^p}}{a^d\Gamma(\frac{d}{p})} \)& \begin{minipage}{0.2 \textwidth} Generalized\\Gamma$(p,a,d)$
\end{minipage} \\  
    \begin{minipage}{0.2 \textwidth} Burr type XII$(c,\,k)$
\end{minipage} & \( \left(1 + x^c \right)^{-k},\,c,\,k>0 \) & \(w_9(x)=\log(1+x^c)\) & \(\frac{ckx^{c-1}}{(1+x^c)^{k+1}}\) & \begin{minipage}{0.2 \textwidth} Burr~type~XII$(c,\,k)$
\end{minipage} \\ 
   \begin{minipage}{0.2 \textwidth} Burr~type~XII$(c,\,k)$
\end{minipage} & \( \left(1 + x^c \right)^{-k},\,c,\,k>0 \) & \(w_{10}(x)=x^a,\,a>0\) & \(\frac{ax^{a-1}(1+x^c)^{-k}}{kB((ck-a)/c,\, (c+a)/c)}\) & \begin{minipage}{0.2 \textwidth} WTRV of\\ Burr type XII$(c,\,k)$
\end{minipage}\\ 
    Kumaraswamy(a,\,b) & \((1-x^a)^b,\,a,\,b>0\) &\(w_{11}(x)=x^a \) & \(abx^{a-1}(1-x^a)^b\) & Kumaraswamy(a,b)\\
     Kumaraswamy(a,b) & \((1-x^a)^b,\,a,\,b>0\) &\(w_{12}(x)=x^c,\,c>0 \) & \(\frac{cx^{c-1}(1-x^a)^b}{b\cdot B(1+c/a,\:b)}\) & \begin{minipage}{0.2 \textwidth} Weighted\\ Kumaraswamy$(a,\,b,\,c)$
\end{minipage} \\\\
    \bottomrule
\end{tabular}}
\end{table}
\section{Reliability Properties of a Random Variable and Its WTRV}
Here we investigate several reliability properties of a random variable \(X\) and its WTRV \(X_w\). Initially, we assume some reliability properties of \(X\) and provide a sufficient condition on the weight function \(w(\cdot)\) to construct \(X_w\) that possesses the ILR [DLR] property. Furthermore, we investigate the sufficient conditions on \(X\) and \(w(\cdot)\) to ensure the preservation of reliability properties of \(X\) by \(X_w\). In particular, we focus on whether aging properties such as IFR, DFR, IMRL, and DMRL of the variable \(X\) are preserved by \(X_w\).  This study may help us to construct a WTRV \(X_w\) that belongs to a prespecified aging class. The significance of studying the aging relationship between \( X \) and \( X_w \) lies in the fact that many reliability properties of \( X_w \) may be directly inferred from those of \( X \) and the property of \(w(\cdot)\). This may facilitate the reliability analysis of some random variables more straightforwardly, despite their survival or failure rate functions not admitting closed-form expressions. \par
The following theorem assumes \(X\) is IFR [DFR] and \(w'(\cdot)\) is log-concave [log-convex] to ensure that \(X_w\) belongs to the ILR [DLR] aging class.
\begin{proposition} \label{proposition1} If
\begin{enumerate}
\item[(a)] $X$ is IFR [DFR], and
\item[(b)] $w'(\cdot)$ is a log-concave [log-convex] function on \(S_X\),
\end{enumerate}
then $X_w$ is ILR [DLR], that is, \(\eta_{X_w}(\cdot)=-f_{X_w}'(\cdot)/f_{X_w}(\cdot)\) is an increasing [decreasing] function.
\end{proposition}
\begin{proof} To prove that \(X_w\) is ILR [DLR], we need to show that \(f_{X_w}(\cdot)\) is a log-concave [log-convex] function on its support \(S_1 = (l_1,\, u_1)\). To proceed, let \( 0<\theta <u_1-l_1 \), and let \( x_1, x_2 \in S_{1} \) satisfy \( l_1 < x_1 < x_2 < u_1-\theta \). Consider the expression
\[
\Delta = f_{X_w}(x_1 + \theta)\, f_{X_w}(x_2) - f_{X_w}(x_1)\, f_{X_w}(x_2 + \theta).
\]
Now, \(\mathbb{E}[w(X)]^2\Delta\) is expressed as
\[ \mathbb{E}[w(X)]^2\Delta= w'(x_1+\theta)\overline{F}_X(x_1+\theta)w'(x_2)\overline{F}_X(x_2)-w'(x_1)\overline{F}_X(x_1)w'(x_2+\theta)\overline{F}_X(x_2+\theta).\]
Given that \(X\) is IFR [DFR], it follows that \( \overline{F}_X(\cdot) \) is log-concave [log-convex] on \(S_X\). Also, \( w'(\cdot) \) is a log-concave [log-convex] function on \(S_X\). It follows from \cref{definition3} that
\[\overline{F}_X(x_1+\theta)\overline{F}_X(x_2)\geq [\leq] \overline{F}_X(x_1)\overline{F}_X(x_2+\theta),\]
and
\[w'(x_1+\theta)w'(x_2)\geq [\leq] w'(x_1)w'(x_2+\theta).\]
Using these inequalities, we have \(\mathbb{E}[w(X)]^2\Delta \geq  [\leq ]\:0\), which further implies that, \(\Delta\geq  [\leq ]\:0\). From \cref{definition3}, we get \(f_{X_w}(\cdot)\) is a log-concave [log-convex] function, that is, \(X_w\) is ILR [DLR].  
\end{proof} \par
The ILR [DLR] attribute of a variable \(X\) is crucial when its failure rate and MRL function do not admit closed-form expressions. In \(1980\), \citet{Glaser01091980} introduced the function \[\eta_X(x)=-\frac{f'_X(x)}{f_X(x)},\quad x \in S_X.\] The Glaser function \( \eta_X(\cdot) \) serves as an effective instrument for revealing the shape of the failure rate and MRL functions for many distributions. Examples include the gamma, beta, and positively truncated normal distributions. Although closed-form expressions for the failure rate and MRL function are not available for these distributions, they admit simple expressions for \( \eta_X(\cdot) \). \citet{Glaser01091980} proved that if \(\eta_X(\cdot)\) of \(X\) is increasing [decreasing], then the failure rate function \(r_X(\cdot)\) is also increasing [decreasing], and the MRL function \(m_X(\cdot)\) of \(X\) is decreasing [ increasing].
\begin{example} Consider \(X\) with SF \( \overline{F}_X(x) = (1 - x)^{\beta - 1}\), \(0 < x < 1\), where \(\beta > 1\).  
Let \(w(x) = x^\alpha\), \(0\leq x<1\),  where \(\alpha > 1\). Then \(r_X(\cdot)\) of \(X\) obtained as   
\[r_X(x) = \frac{\beta-1}{1-x}, \quad 0\leq x< 1.\]  
 As \(r_X(\cdot)\) is increasing, \(X\) is IFR. Now, we have
\[\frac{d^2}{dx^2}\bigg(\log w'(x)\bigg) = \frac{1-\alpha}{x^2} < 0,\]  
 that is, \( w'(\cdot) \) is a log-concave function. The PDF of the WTRV \(X_w\)  is (see \cref{table1}) defined as
\[f_{X_w}(x) = \frac{\Gamma(\alpha + \beta)}{\Gamma(\alpha) \Gamma(\beta)} x^{\alpha - 1} (1 - x)^{\beta - 1}, \quad 0 < x < 1.\]  
The Glaser function \(\eta_{X_w}(\cdot)\) of \(X_w\) is obtained as   
\[\eta_{X_w}(x) = \frac{\beta-1}{1-x} - \frac{\alpha-1}{x},\quad  0 < x < 1.\] Since \(\eta_{X_w}(\cdot)\) is an increasing function, \(X_w\) exhibits the ILR property. 
\end{example} \par
Setting \( w(x) = x \) in \cref{proposition1}, we have \( w'(x) = 1 \), which is trivially both log-concave and log-convex. This yields the following corollary.
\begin{corollary} Let \(X\) be a non-negative continuous random variable. Then, \(X\) is IFR [DFR] if and only if \(\widetilde{X}\) is ILR [DLR].
\begin{proof} The proof follows from the expression of the Glaser function of \(\widetilde{X}\) as \(\eta_{\widetilde{X}}(\cdot)=r_X(\cdot)\).
\end{proof}
\end{corollary} \par
We now consider \( X \) is IFR and establish conditions on \( w(\cdot) \) and \(r_X(\cdot)\) that ensure \( X_w \) preserves the IFR property.
\begin{theorem} \label{theorem1}  Let \(X\) is IFR. If the function \( w'(\cdot)/r_X(\cdot) \) is both increasing and log-concave on \( S_X \), then \( X_w \) is also IFR.
\end{theorem} 
\begin{proof} Since \(w(\cdot)\) increasing function on \( (0,\, u_X)\), it follows that \( u_X = u_1 \). For a fixed \( 0 < \theta < u_1 - l_1 \), and for \( x_1 \) and \( x_2 \) such that \( l_1 - \theta < x_1 < x_2 < u_1 = u_X \), define 
\[
\Delta_1 = \overline{F}_{X_w}(x_1+\theta)\overline{F}_{X_w}(x_2) - \overline{F}_{X_w}(x_1)\overline{F}_{X_w}(x_2+\theta).\]
Then, \(\mathbb{E}[w(X)]^2 \Delta_1 \) is expressed as \begin{align*} \mathbb{E}[w(X)]^2 \Delta_1 &=  \left[ \int_{x_1+\theta}^\infty w'(z_1) \overline{F}_X(z_1)dz_1 \right] \left[ \int_{x_2}^\infty w'(z_2) \overline{F}_X(z_2)dz_2 \right] \\
& \qquad- \left[ \int_{x_1}^\infty w'(z_2) \overline{F}_X(z_2)dz_2 \right] \left[ \int_{x_2+\theta}^\infty w'(z_1) \overline{F}_X(z_1)dz_1 \right]\\
&=  \left[ \int_{x_1}^{u_X-\theta} \frac{w'(z_1+\theta)}{r_X(z_1+\theta)} f_X(z_1+\theta)dz_1 \right] \left[ \int_{x_2}^{u_X}  \frac{w'(z_2)}{ r_X(z_2)}f_X(z_2)dz_2 \right] \\
& \qquad- \left[ \int_{x_1}^{u_X} \frac{w'(z_2)}{ r_X(z_2)}f_X(z_2)dz_2 \right] \left[ \int_{x_2}^{u_X-\theta} \frac{w'(z_1+\theta)}{r_X(z_1+\theta)} f_X(z_1+\theta)dz_1 \right].\\
 \end{align*} 
Consider the independent variables, \( Z_1 \) and \( Z_2 \), with PDFs \( f(z_1 + \theta) \) and \( f(z_2) \), respectively. Let \( K_i(\cdot) \) represent the CDFs of \( Z_i \), where
 \begin{align*} l_{Z_i} = \inf\{z \in \mathbb{R} : K_i(z) > 0\}, \:\:u_{Z_i} = \sup\{z \in \mathbb{R} : K_i(z) < 1\}, \:\: S_{Z_i} = (l_{Z_i}, u_{Z_i}), \:\: i = 1, 2. \end{align*}
Then the supports of \( Z_1 \) and \( Z_2 \) are given by \( S_{Z_1} = (l_X - \theta,\, u_X - \theta) \) and \( S_{Z_2} = (l_X,\, u_X) \), respectively. Given that \( X \) exhibits the IFR property, it follows that  
\[
r_{Z_1}(x) = \frac{f_X(x + \theta)}{\overline{F}_X(x + \theta)} = r_X(x + \theta) \geq r_X(x) = \frac{f_X(x)}{\overline{F}_X(x)} = r_{Z_2}(x),
\]  
which implies that \( Z_1 \le_{fr} Z_2 \). Now we verify the conditions (i)-(iii) of \cref{lemma2} for \(Z_1\) and \(Z_2\) to show that \(\Delta_1\geq 0\). Since \(l_X - \theta < x_1 < x_2 <  u_X - \theta <u_X  \), it follows that  
\[
S_{Z_1} \cap S_{Z_2} = (l_X - \theta, u_X - \theta)  \cap (l_X, u_X) = (l_X, u_X-\theta) \neq \emptyset.
\]
Now define \[\phi_2(x, y) = \frac{w'(x+\theta) w'(y )}{r_X(x+\theta)r_X(y)} I(x_1 \le x < u_X-\theta) I(x_2 \le y < u_X ), \] and
\[ \phi_1(x, y) = \frac{w'(x+\theta) w'(y )}{r_X(x+\theta)r_X(y)} I(x_2 \le x < u_X-\theta) I(x_1 \le y < u_X),\]
where \((x,\,y)\in S_{Z_1}\times S_{Z_2}\). Then, utilizing the independence of \(Z_1\) and \(Z_2\), \( \mathbb{E}[w(X)]^2\Delta_1\) is expressed as
\begin{align*} \mathbb{E}[w(X)]^2 \Delta_1 = & E\left[ \phi_2(Z_1,\,Z_2) \right]  - E\left[ \phi_1(Z_1,\,Z_2) \right]. \end{align*}
For $(x, y) \in S_{Z_1} \times S_{Z_2}$, let \(\Delta\phi_{21}(x, y) = \phi_2(x, y) - \phi_1(x, y)\), then
\begin{equation*} \Delta\phi_{21}(x, y) = \begin{cases} \dfrac{w'(x+\theta) w'(y )}{r_X(x+\theta)r_X(y)}, & \text{if } x_1 \le x < x_2 \text{ and } x_2 \le y < u_X, \\[10pt] -\dfrac{w'(x+\theta) w'(y )}{r_X(x+\theta)r_X(y)}, & \text{if } x_2 \le x < u_X-\theta \text{ and } x_1 \le y < x_2, \\ 0, & \text{otherwise}. \end{cases} \end{equation*} 
It is evident that \( \Delta\phi_{21}(x, y) \ge 0 \) for all \( (x, y) \in S_{Z_1} \times S_{Z_2} \) with \( x \le y \). Using the non-negativity and increasing property of \( w'(\cdot)/r_X(\cdot) \) on \( S_{Z_2} = S_X \), we have for each fixed \( x \in S_{Z_1} \cap S_{Z_2} = (l_X,\, u_X - \theta) \), \( \Delta\phi_{21}(x, y) \) is an increasing function in \( y \in [x, u_{Z_2}] = [x, u_X] \). Further, for all \( (x,\, y) \in (S_{Z_1} \cap S_{Z_2}) \times (S_{Z_1} \cap S_{Z_2}) \) with \( x \leq y \), using the log-concave propery of \( w'(\cdot)/r_X(\cdot) \) over \( S_{X} \) we get
\[\Delta\phi_{21}(x, y) +\Delta\phi_{21}(y, x)=\frac{w'(x+\theta) w'(y )}{r_X(x+\theta)r_X(y)}-\frac{w'(x )w'(y+\theta) }{r_X(x)r_X(y+\theta)}\geq 0.\]
 Thus, the functions \( \phi_1(\cdot, \cdot) \), \( \phi_2(\cdot, \cdot) \), and \(\Delta\phi_{21}(\cdot, \cdot)\) satisfy conditions (i)–(iii) of \cref{lemma2}, with the independent random variable \(Z_1\) and \(Z_2\) satisfying \( Z_1 \le_{\text{fr}} Z_2 \). Hence, by applying \cref{lemma2}, one can deduce that \(\mathbb{E}[\phi_2(Z_1, Z_2)] \ge \mathbb{E}[\phi_1(Z_1, Z_2)]\). Consequently, we have
\[
\Delta_1 = \overline{F}_{X_w}(x_1 + \theta)\,\overline{F}_{X_w}(x_2) - \overline{F}_{X_w}(x_1)\,\overline{F}_{X_w}(x_2 + \theta) \ge 0,
\]  
which establishes the log-concavity of \( \overline{F}_{X_w}(\cdot) \) on \( S_{X_w} \) by using \cref{definition3}. Therefore, \( X_w \) possesses the IFR property.  
\end{proof}
\begin{example} 
Let \(X\) follow the standard uniform distribution. Then failure rate \(r_X(\cdot) \) is defined by
\[
r_X(x) = \dfrac{1}{1 - x}, \quad 0 < x < 1.
\]
Let \(w(x) = -\log(1 - x^2)\), \(0<x<1\). Then, 
\[
\dfrac{w'(x)}{r_X(x)} = \frac{2x(1 - x)}{1 - x^2} = \dfrac{2x}{1 + x},\quad 0<x<1.
\]
 It is an increasing function. Moreover, the second derivative 
\[
\dfrac{d^2}{dx^2} \left( \log \left( \dfrac{w'(x)}{r_X(x)} \right) \right) = -\frac{1}{x^2} + \frac{1}{(1 + x)^2} < 0\quad \text{for\quad} 0 < x < 1.
\]
Therefore, \(w'(\cdot)/r_X(\cdot)\) is log-concave over \(S_X=(0,\,1)\). Hence, by applying \cref{theorem1}, we conclude that \(X_w\) exhibits the IFR property.
\end{example} \par
Now we consider \( X \) is DFR and establish conditions on \( w(\cdot) \) and \( r_X(\cdot) \) ensuring that \( X_w \) is also DFR.
\begin{theorem} \label{theorem2} Let \( X \) is DFR. If the function \( w'(\cdot)/r_X(\cdot) \) is increasing and log-convex on \( S_X \), then \( X_w \) is also DFR.
\end{theorem}
\begin{proof} Since \(X\) exhibits DFR property, the upper bound of \(X\) is infinity, i.e., \(u_X = \infty\), which implies that \(S_X = (l_X,\, \infty)\). Furthermore, from the increasing property of  \(w(\cdot)\) over \((l_X,\, \infty)\), we get the support of \(X_w\) is \(S_{X_w} = (l_1,\, \infty)\). For a fixed \(\theta > 0\) and for \(x_1,\,x_2\) satisfying \(l_1 < x_1 < x_2 < \infty\), consider the following expression:
\[
\Delta_2 = \overline{F}_{X_w}(x_1)\, \overline{F}_{X_w}(x_2 + \theta) - \overline{F}_{X_w}(x_1 + \theta)\, \overline{F}_{X_w}(x_2).
\]
Then, \(l_X\leq l_1<x_1<x_2<\infty\). After some manipulations, we obtain
\begin{align*}
\mathbb{E}[w(X)]^2 \Delta_2&=\left[ \int_{x_1}^{\infty}  \frac{w'(z_1)}{ r_X(z_1)}f_X(z_1)dz_1 \right]\left[ \int_{x2}^{\infty} \frac{w'(z_2+\theta)}{r_X(z_2+\theta)} f_X(z_2+\theta)dz_2 \right] \\
& \qquad- \left[ \int_{x_1}^{\infty} \frac{w'(z_2+\theta)}{ r_X(z_2+\theta)}f_X(z_2+\theta)dz_2 \right] \left[ \int_{x_2}^{\infty} \frac{w'(z_1)}{r_X(z_1)} f_X(z_1)dz_1 \right]\\
\end{align*}
Let \(Z_1\) and \(Z_2\) be two independent random variables whose PDFs are given by
\(
f_{Z_1}(z_1) = f_X(z_1) \quad \text{for } z_1 > l_X,
\)
and
\(
f_{Z_2}(z_2) = f_X(z_2 + \theta) \quad \text{for } z_2 > l_X - \theta,
\)
respectively.  Denote by \( K_i(\cdot) \) the CDFs of \( Z_i \), \(i=1,2\) such that
 \begin{align*} l_{Z_i} = \inf\{z_i \in \mathbb{R} : K_i(z_i) > 0\}, \:\:u_{Z_i} = \sup\{z_i \in \mathbb{R} : K_i(z_i) < 1\}, \:\: S_{Z_i} = (l_{Z_i}, u_{Z_i}), \:\: i = 1, 2. \end{align*}
 Then \(S_{Z_1}\cap S_{Z_2}=S_X=(l_X,\,\infty)\). Since \(X\) has DFR property, it follows that \(Z_1\leq_{fr}Z_2\). Now using the independence of \(Z_1\) and \(Z_2\) we have
 \begin{align*}\mathbb{E}[w(X)]^2 \Delta_2&=\mathbb{E}\left[ \frac{w'(Z_1)}{ r_X(Z_1)}\frac{w'(Z_2+\theta)}{r_X(Z_2+\theta)} I(Z_1\geq x_1)I(Z_2\geq x_2)\right] \\
 &\qquad- \mathbb{E}\left[  \frac{w'(Z_2+\theta)}{ r_X(Z_2+\theta)}  \frac{w'(Z_1)}{r_X(Z_1)} I(Z_1\geq x_2)I(Z_2\geq x_1)\right].
 \end{align*}
 Define \[\phi_2(x, y) = \frac{w'(x) w'(y+\theta )}{r_X(x)r_X(y+\theta)} I(x\geq x_1) I(y\geq x_2 ), \] and
\[ \phi_1(x, y) = \frac{w'(x) w'(y+\theta )}{r_X(x)r_X(y+\theta)} I(x\geq x_2) I(y\geq x_1),\]
where \((x,\,y)\in S_{Z_1}\times S_{Z_2}\). Then, \(\mathbb{E}[w(X)]^2 \Delta_2\) is expressed as
\begin{align*} \mathbb{E}[w(X)]^2 \Delta_2 = E\left[ \phi_2(Z_1,\,Z_2) \right]  - E\left[ \phi_1(Z_1,\,Z_2) \right]. \end{align*}
For \((x,\,y)\in S_{Z_1}\times S_{Z_2}\), define \(\Delta\phi_{21}(x,\,y)= \phi_2(x,\,y)-\phi_1(x,\,y) \), then we have
\[\Delta\phi_{21}(x,\,y)=\begin{cases}
     \dfrac{w'(x)w'(y+\theta)}{r_X(x)r_X(y+\theta)}, & \text{\quad if\quad} x_2 \geq x\geq x_1 \text{\quad and \quad} y\geq x_2,\\[10pt]
-\dfrac{w'(x)w'(y+\theta)}{r_X(x)r_X(y+\theta)}, & \text{\quad if\quad} x\geq x_2 \text{\quad and \quad} x_2 \geq y\geq x_1,\\
0, & \text{\quad otherwise.}
\end{cases}\]
Observe that the intersection of the supports is \( S_{Z_1} \cap S_{Z_2} = (l_X,\, \infty) \neq \emptyset \). We now verify conditions (i)-(iii) of \cref{lemma2} for the independent variable \(Z_1\) and \(Z_2\).  From the definition of \( \Delta\phi_{21}(\cdot, \cdot)  \), we have \( \Delta\phi_{21}(x, y) \geq 0 \) for all \( (x, y) \in S_{Z_1} \times S_{Z_2} \) such that \( x \leq y \). Since \(w'(\cdot)/r_X(\cdot)\) is  increasing function on \(S_{X}\), for each fixed \(x\in S_{Z_1}\cap S_{Z_2}\), \(\Delta\phi_{21}(x,\,y)\) is increasing in \(y \in [x,\,\infty)\). Moreover, for \( x \leq y \), the log-convexity of \(w'(\cdot)/r_X(\cdot)\) on \( S_X \) implies that
\[\Delta\phi_{21}(x, y) +\Delta\phi_{21}(y, x)=\frac{w'(x) }{r_X(x)}\cdot\frac{ w'(y+\theta )}{r_X(y+\theta)}-\frac{w'(x+\theta )}{r_X(x+\theta))}\cdot\frac{w'(y)}{r_X(y)}\geq 0,\]
for \((x,\,y)\in (S_{Z_1}\cap S_{Z_2})\times (S_{Z_1}\cap S_{Z_2})\). Now, since \(Z_1\) and \(Z_2\) are independent and \( Z_1 \leq_{fr} Z_2 \), applying \cref{lemma2} to \(Z_1\) and \(Z_2\), it follows that
\begin{align*} \mathbb{E}[w(X)]^2 \Delta_2 = E\left[ \phi_2(Z_1,\,Z_2) \right]  - E\left[ \phi_1(Z_1,\,Z_2) \right]\geq 0. \end{align*}
Consequently,
\[
\Delta_2 = \overline{F}_{X_w}(x_1)\, \overline{F}_{X_w}(x_2 + \theta) - \overline{F}_{X_w}(x_1 + \theta)\, \overline{F}_{X_w}(x_2)\geq 0.
\] From \cref{definition3}, we get \(\overline{F}_{X_w}(\cdot)\) is log-convex on \(S_{X_w}\). Hence, \(X_w\) exhibits the DFR property. 
\end{proof}
\begin{example} Let \(X\) follows Pareto distribution with SF \[\overline{F}_X(x)=(1+x)^{-\alpha},\quad x\geq 0,\quad \alpha>0.\]
Let \(w(x)=e^{(x+1)^2}-e,\quad x\geq 0\). Then \(w'(\cdot)\) and \(r_X(\cdot)\) are defined by 
\[w'(x)=2(x+1)e^{(x+1)^2},\quad x\geq0, \text{\quad and\quad}r_X(x)=\frac{\alpha}{x+1}, \quad x\geq 0,\]
respectively. We observe that \[\frac{w'(x)}{r_X(x)}=\frac{2(x+1)^2e^{(x+1)^2}}{\alpha},\quad x\geq 0,\] is increasing. Furthermore, observed that
\[\dfrac{d^2}{dx^2}\log\left(\frac{w'(x)}{r_X(x)}\right)=2-\frac{2}{(x+1)^2}\geq 0\quad \text{for}\quad x\geq 0,\]
which ensures the log-convexity of \(w'(\cdot)/r_X(\cdot)\). Hence, by applying \cref{theorem2}, we conclude that \(X_w\) exhibits the DFR property.
\end{example} \par
We now consider \( X \) is DMRL and establish conditions on \( w(\cdot) \), \( r_X(\cdot) \), and the MRL function \( m_X(\cdot) \) to ensure that \( X_w \) has the DMRL property.
\begin{theorem} \label{theorem3} Let \( X \) is DMRL. If \( w'(\cdot)/r_X(\cdot) \) is increasing and log-concave on \( S_X \) and \( m_X(\cdot) \) is log-convex on \( S_X \), then \( X_w \) is IFR and hence DMRL.
\end{theorem}
\begin{proof} By applying Theorem~2.5(a) of \citet{Misra31012008}, which assumes the DMRL property of \( X \) and the log-convex property of \( m_X(\cdot) \) on \( S_X \), we get that \( X \) has the IFR property. Now, using the IFR property of \( X \) and increasing and log-concave properties of \( w'(\cdot)/r_X(\cdot) \) on \( S_X \) to \cref{theorem1}, we get \( X_w \) is IFR and hence DMRL.  
\end{proof} \par
Now consider \( X \) is IMRL and establish conditions on \( w(\cdot) \), \( r_X(\cdot) \), and \( m_X(\cdot) \) to ensure that \( X_w \) is also IMRL.

\begin{theorem}\label{theorem4} Let \( X \) is IMRL. If \(w'(\cdot)/r_X(\cdot)\) is increasing and log-convex on \(S_X\), and the function \(m_X(\cdot)\) is log-concave on \(S_X\), then \(X_w\) is DFR and hence IMRL. 
\end{theorem}
\begin{proof}   By applying Theorem~2.5(b) of \citet{Misra31012008}, which assumes the IMRL property of \( X \) and the log-concave property of \( m_X(\cdot) \) on \( S_X \), we get that \( X \) has the DFR property. Now, using the DFR property of \( X \) and the increasing and log-convex properties of \( w'(\cdot)/r_X(\cdot) \) on \(S_X\) to \cref{theorem2}, we conclude that \( X_w \) is DFR and hence IMRL.  
\end{proof}
%%%%%%%%%%%%%%%%%%%%%%%%%%%%%%%%%%%%%%%%%%%%%%%%%%%%%%%%%%%%%%%%%%%%%%%%%%%%%%%%%%%%%%%%%%%%%%%%%%%%%%%%%%%%%%%%%%%%%%%
%%%%%%%%%%%%%%%%%%%%%%%%%%%%%%%%%%%%%%%%%%%%%%%%%%%%%%%%%%%%%%%%%%%%%%%%%%%%%%%%%%%%%%%%%%%%%%%%%%%%%%%%%%%%%%%%%%%%%%%
\section{Stochastic Ordering Properties of Random Variables and Their WTRVs}   
 This section focuses on several stochastic ordering properties involving \( X \) and its WTRV \( X_w \). At first, we establish conditions on \(X\), \(Y\), \(w_1(\cdot)\), and \(w_2(\cdot)\) under which the WTRVs \(X_{w_1}\) and \(Y_{w_2}\) are stochastically ordered by likelihood ratio ordering. We also establish the conditions under which the likelihood ratio ordering holds between \(X\) and \(X_w\) as well as between \(\widetilde{X}\) and \(X_w\). We then establish the usual stochastic ordering between \(X\) and \(Y\) when their WTRVs \(X_{w_1}\) and \(Y_{w_2}\) are stochastically ordered by reversed failure rate ordering. Furthermore, we examine and find conditions under which various stochastic orderings between two random variables \(X\) and \(Y\) are preserved between their respective WTRVs \(X_w\) and \(Y_w\). We consider the orderings, such as the usual stochastic order, failure rate order, and reversed failure rate order, for this purpose. This study may provide a tool to construct WTRVs for \(X\) and \(Y\), when some ordering property between \(X_w\) and \(Y_w\) is prespecified. \par
We first assume the failure rate order between \( X \) and \( Y \), and conditions on \( w_1(\cdot) \) and \( w_2(\cdot) \) that ensure the likelihood ratio order between \( X_{w_1} \) and \( Y_{w_2} \). We present the following theorem.
\begin{theorem} \label{theorem5}\begin{enumerate}[(i)]
    \item  If \:\(l_1 \leq l_2\),  \(u_1 \leq u_2\), \(X \leq_{fr} Y\), \(\frac{w_2'(\cdot)}{w_1'(\cdot )}\) is increasing in on  \(S_1\cap S_2\), with \(w_1'(x) \neq 0\) for all \(x \in S_1\), then \(X_{w_1} \leq_{lr} Y_{w_2}\).\\[1pt]
\item[(ii)]  If \:\(S_X=S_1\), \(S_Y= S_2\), \( X_{w_1} \leq_{lr} Y_{w_2} \), and \(w_1'(\cdot)/w_2'(\cdot)\) is increasing function on \(S_1\cap S_2\), with \( w_2'(x) \neq 0 \) for all \( x \in S_2 \), then \( X \leq_{fr} Y \). 
\end{enumerate} 
 
\end{theorem}
\begin{proof} (i) Given that \(l_1 \leq l_2\) and \(u_1 \leq u_2\). For \(x\in S_1\cap S_2\), consider the following relation obtained by using the PDFs of \(X_{w_1}\) and \(Y_{w_2}\): 
\begin{align}
\frac{f_{Y_{w_2}}(x)}{f_{X_{w_1}}(x)} &= \frac{w_2'(x) \overline{F}_Y(x)}{w_1'(x) \overline{F}_X(x)} \cdot \frac{\mathbb{E}[w_1(X)]}{\mathbb{E}[w_2(Y)]}. \label{7}
\end{align}  
Here, \( X \leq_{fr} Y \) implies that \( \overline{F}_Y(\cdot)/\overline{F}_X(\cdot) \) is increasing on \( S_X \cap S_Y \), and hence on \( S_1 \cap S_2 \), since \( S_1 \cap S_2 \subset S_X \cap S_Y \). Now, it is given that \( w_2'(\cdot)/w_1'(\cdot)\) is increasing on \(S_1\cap S_2\) and \(w_1'(x) \neq 0\) for all  \(x\in S_1\). Hence, \(f_{Y_{w_2}}(\cdot)/f_{X_{w_1}}(\cdot)\) is increasing on \(S_1\cap S_2\), that is, \(X_{w_1} \leq_{lr}  Y_{w_2}\).\\[5pt]
(ii) The proof is similar to part (i) and follows directly from the relation below, obtained from \eqref{7}.
\begin{align}  
\frac{\overline{F}_Y(x)}{\overline{F}_X(x)} &= \frac{f_{Y_{w_2}}(x)}{f_{X_{w_1}}(x)} \cdot \frac{w_1'(x)}{w_2'(x)} \cdot \frac{\mathbb{E}[w_2(Y)]}{\mathbb{E}[w_1(X)]}. \label{8}  
\end{align}
\end{proof}
\begin{example} (i)\: Let \( X \) and \( Y \) follow exponential distributions with rate parameters \( \lambda_1 \) and \( \lambda_2 \), respectively, where \( 0 < \lambda_2 < \lambda_1 \). Consider \( w_1(x) = x^{k_1} \) and \( w_2(x) = x^{k_2} \), \(x\geq 0\), where \( 0 < k_1 < k_2 \).  We first observe that  
\[
\frac{\overline{F}_Y(x)}{\overline{F}_X(x)} = e^{(\lambda_1 - \lambda_2)x},\: x>0 \quad \text{and} \quad \frac{w_2'(x)}{w_1'(x)} = x^{k_2 - k_1}, \quad x > 0,
\]  
are increasing functions of \( x \). Consequently, \( X \leq_{fr} Y \) and \( w_2'(\cdot)/w_1'(\cdot) \) is increasing. Observe that, \(l_1=l_2=0\) and \(u_1=u_2=\infty\). Now, by applying \cref{theorem5}, it follows that \( X_{w_1} \leq_{lr} Y_{w_2} \).
\end{example}
In \cref{theorem5}, if we set \( w_1(x) = w_2(x) = x \), then the WTRVs \( X_{w_1} \) and \( Y_{w_2} \) becomes \( \widetilde{X} \) and \( \widetilde{Y} \) of \(X\) and \(Y\), respectively. Then we have the following corollary.
\begin{corollary} 
Suppose \( X \) and \( Y \) are non-negative random variables. \( X \leq_{fr} Y \) if and only if  \(  \widetilde{X} \leq_{lr} \widetilde{Y} \). 
\end{corollary}
\begin{remark}  
The above corollary was proved by \citet{gupta2007role}.
\end{remark} \par
We now consider \(X\) is IFR [DFR] and assume conditions on \( w(\cdot) \) and \( r_X(\cdot) \) to establish a likelihood ratio order between \( X \) and \( X_w \). We state the following proposition.
\begin{proposition}\label{proposition2} Let \(X\) be defined on \(S_X=S_1=(l_1,\,u_1)\). If 
\begin{enumerate}
\item[(i)] \( X \) is IFR [DFR], and 
\item[(ii)] \( w(\cdot) \) is a strictly increasing and concave [convex] function on \(S_X\), \end{enumerate}
then, \( X_w \leq_{lr} X \) [\( X \leq_{lr} X_w \)].
\end{proposition}
\begin{proof} Since \(S_X = S_1\), it follows that \(l_X = l_1\) and \(u_X = u_1\). Now, for \(x \in S_X \cap S_1\), we consider
\[ \frac{f_{X_w}(x)}{f_X(x)}=\frac{\overline{F}_X(x)\cdot w^{\prime}(x)}{f_X(x)\cdot \mathbb{E}[w(X)]}=\frac{1}{r_X(x)} \times \frac{w^{\prime}(x)}{\mathbb{E}[w(X)]}.\] 
The IFR [DFR] property of \(X\) implies that \(1/r_X(\cdot)\) is a decreasing [increasing] on \(S_X\). Using the strictly increasing and concave [convex] properties of \(w(\cdot)\), we get \(w'(\cdot)\) is decreasing [increasing] on \(S_X\). Therefore, the function \( f_{X_w}(\cdot)/f_X(\cdot)\) is decreasing [increasing] on \(S_X\cap S_1\), implying that \(X_w \leq_{lr} X\) [\(X\leq_{lr}X_w\)].  
\end{proof}
\begin{example}  
Let \( X \) has CDF \( F(x) = 1 - e^{-x^2} \), \(x>0\). Consider \( w(x) = x^k \), \(x>0\), where \( 0 < k < 1 \). Then \(r_X(\cdot)\) is defined as \[ r_X(x) = 2x,\quad x> 0. \]  
 Therefore, \( X \) is IFR. Now for \( 0 < k < 1 \), we obtain \[ w''(x) = k(k-1)x^{k-2} < 0 \quad \text{for}\quad x > 0. \]
So, \( w(\cdot) \) is a strictly increasing and concave. The PDFs of \( X \) and its WTRV \( X_w \) are 
\[
f_X(x) = 2x e^{-x^2} \quad \text{and} \quad  
f_{X_w}(x) = \frac{kx^{k-1} e^{-x^2}}{\Gamma(\frac{k}{2} + 1)}, \quad x \geq 0,
\]  
respectively. Now, by applying \cref{proposition2}, we get \(X_w \leq_{lr} X\).
\end{example}
In \cref{proposition2}, If we take \( w(x) = x \), the WTRV \( X_w \) becomes \( \widetilde{X}\). We present the following corollary.
\begin{corollary}  \(X\) is IFR [DFR] if and only if \(\widetilde{X}\leq_{lr} X\) \([ X \leq_{lr}\widetilde{X}] \)
\end{corollary}
\begin{remark} The proof of the above corollary follows from \( f_X(\cdot)/f_{\widetilde{X}}(\cdot) = \mathbb{E}[X] \cdot r_X(\cdot) \); (see \citet{bon2005ageing}).
\end{remark} \par
We now consider conditions on \( X \) and \( w(\cdot) \) to establish a likelihood ratio order between \(\widetilde{X}\) and \( X_w \). We state the following proposition.
\begin{proposition} \label{proposition3} Let \( \widetilde{X} \) be defined on \( S_{\widetilde{X}}=S_1 = (l_1,\,u_1)\). Then,  
\[
\widetilde{X} \leq_{lr} X_w \quad [X_w \leq_{lr} \widetilde{X}] \quad \text{ if and only if\: } w(\cdot) \text{ is a convex [concave] function on } S_1.
\]  
\end{proposition}
\begin{proof}  From \(S_{\widetilde{X}}=S_1\), we get \(l_{\widetilde{X}}=l_1 \) and \( u_{\widetilde{X}}=u_1\). The PDF \(f_{\widetilde{X}}(\cdot)\) is defined by Equation~\eqref{2}. For \(x\in S_{\widetilde{X}}\cap S_1\), consider 
\[
\frac{f_{X_w}(x)}{f_{\widetilde{X}}(x)} = \frac{w'(x) \overline{F}_X(x)}{\mathbb{E}[w(X)]} \times \frac{\mathbb{E}[X]}{\overline{F}_X(x)} = \frac{w'(x) \mathbb{E}[X]}{\mathbb{E}[w(X)]}.
\]  
 Now, \(f_{X_w}(\cdot)/f_{\widetilde{X}}(\cdot)\) is an increasing [decreasing] on \(S_{\widetilde{X}}\cap S_1\) if and only if \( w'(\cdot)\) is increasing [decreasing] on \(S_1\). Consequently, \(\widetilde{X} \leq_{lr} X_w \quad [X_w \leq_{lr} \widetilde{X}]\) if and only if \(w(\cdot)\) is convex [concave] on \(S_1\).    
\end{proof} \par
The theorem below provides the conditions on \(w_1(\cdot)\) and \(w_2(\cdot)\) so that the reversed failure rate ordering between \(X_{w_1}\) and \(Y_{w_2}\) is transverse to the usual stochastic ordering between \(X\) and \(Y\).
\begin{theorem}  \label{theorem6} Let \( X \) and \(Y\) be defined on \(S_X=S_1=(l_1,\,u_1)\) and \(S_Y=S_2=(l_2,\,u_2)\), respectively. If \(X_{w_1} \leq_{rfr} Y_{w_2}\), \( \frac{w_1'(\cdot)}{w_2'(\cdot)}\) is an increasing function, and \(w_2'(x),\, w_1'(0) \neq 0 \) on  \( S_X\cap S_Y\), then  \(X \leq_{st} Y \).  
\end{theorem}
\begin{proof} Since \( X_{w_1} \leq_{rfr} Y_{w_2} \), \(S_X=S_1\) and \(S_Y=S_2\), we have \(l_X=l_1\leq l_2=l_Y\) and \(u_X=u_1\leq u_2=u_Y\). Now, \( X_{w_1} \leq_{rfr} Y_{w_2} \) implies that \(\zeta(\cdot)\) defined by
\[
\zeta(t) = \frac{\int_0^{t} w_2'(x)\, \overline{F}_Y(x)\,dx}{\int_0^{t} w_1'(x)\, \overline{F}_X(x)\,dx}, \quad t \in S_1 \cap S_2
\]
is increasing in \(t\). Thus, for \( t \geq z>0\), we have \( \zeta(t) \geq \zeta(z) \), which gives  
\begin{align}  
\frac{\int_0^{t} w_2'(x) \overline{F}_Y(x)dx}{\int_0^{t} w_1'(x) \overline{F}_X(x)dx} \geq \frac{\int_0^{z} w_2'(x) \overline{F}_Y(x)dx}{\int_0^{z} w_1'(x) \overline{F}_X(x)dx}. \label{9}  
\end{align}  
Since \( \zeta(\cdot) \) is increasing, we have \(\zeta^\prime(t) \geq 0 \) for \(t\in S_1\cap S_2\)
\begin{align}  
\text{ or,\quad } & \frac{w_2'(t) \overline{F}_Y(t) \int_0^{t} w_1'(x) \overline{F}_X(x)dx - w_1'(t) \overline{F}_X(t) \int_0^{t} w_2'(x) \overline{F}_Y(x)dx}{\bigg(\int_0^{t} w_1'(x) \overline{F}_X(x)dx\bigg)^2} \geq 0 \nonumber \\  
or,\quad& \frac{w_2'(t) \overline{F}_Y(t)}{w_1'(t) \overline{F}_X(t)} \geq \frac{\int_0^{t} w_2'(x) \overline{F}_Y(x)dx}{\int_0^{t} w_1'(x) \overline{F}_X(x)dx}. \label{10}  
\end{align}  
Using inequalities \eqref{9} and \eqref{10}, we obtain  
\begin{equation}  
\frac{w_2'(t) \overline{F}_Y(t)}{w_1'(t) \overline{F}_X(t)} \geq \lim_{z \to 0} \frac{\int_0^{z} w_2'(x) \overline{F}_Y(x)dx}{\int_0^{z} w_1'(x) \overline{F}_X(x)dx}. \label{11}  
\end{equation}  
Since the right-hand side of the inequality \eqref{11} is expressed as \( \frac{0}{0} \), the application of L'Hôpital's rule yields 
\begin{align}  
\frac{w_2'(t) \overline{F}_Y(t)}{w_1'(t) \overline{F}_X(t)} &\geq \frac{w_2'(0) \overline{F}_Y(0)}{w_1'(0) \overline{F}_X(0)} = \frac{w_2'(0)}{w_1'(0)}, \nonumber \\  
or,\qquad \frac{\overline{F}_Y(t)}{\overline{F}_X(t)} &\geq \frac{w_1'(t)}{w_2'(t)} \cdot \frac{w_2'(0)}{w_1'(0)}. \nonumber 
\end{align}  
Using the fact that \( w_1'(\cdot)/w_2'(\cdot)\) is increasing on \(S_X \cap S_Y\), we obtain \( \overline{F}_Y(t) \geq \overline{F}_X(t) \) for all \( t\in S_1\cap S_2=S_X\cap S_Y \). Hence, \(X \leq_{st} Y\).   
\end{proof}  
\begin{example}  
Let \( X \) and \( Y \) have CDFs \( F_X(x) = 1 - e^{-2x}, \quad x>0 \text{\: and \:} F_Y(x) = 1 - e^{-x},\quad x>0,\) respectively. Let \(w_1(x) = x^2\) and \(w_2(x) = x\), \(x > 0\). Then observe that  
\[
\frac{w_1'(x)}{w_2'(x)} = 2x,\quad x> 0,
\]  
 is increasing. Furthermore, the function
\begin{equation*}  
 \frac{F_{Y_{w_2}}(x)}{F_{X_{w_1}}(x)} = 2e^{2x} - 2(1 + x)e^{x}, \quad x > 0 
\end{equation*}
is increasing on \( (0,\, \infty) \). Therefore, \( X_{w_1} \leq_{rfr} Y_{w_2} \). Here, \( S_X = S_1 = S_Y = S_2 = (0,\, \infty) \). Applying \cref{theorem6}, we conclude that \( X \leq_{st} Y \).

\end{example}
If we set \( w_1(x) = w_2(x) = x \) in \cref{theorem6}, then the WTRVs \( X_{w_1} \) and \( Y_{w_2} \) becomes \( \widetilde{X} \) and \( \widetilde{Y} \), respectively. We state the following corollary.
\begin{corollary} If \( \widetilde{X} \leq_{rfr}\widetilde{Y} \), then \(X\leq_{st} Y\).
\end{corollary}
\begin{remark}  The above corollary was proved by \citet{li2008reversed}.
\end{remark} \par 
We now consider the minimum of two independent variable \( X \) and \( Y \), denoted by \( X \land Y \), and establish a likelihood ratio order between \( X_w \land Y_w \) and \( (X \land Y)_w \), the WTRV of \( X \land Y \).
\begin{theorem}  \label{theorem7}Let \( X \) and \( Y \) be independent with support \( S_X = S_Y \), and let \( X_w \) and \( Y_w \) also be independent. If\: \(X_w \leq_{fr} \;[\geq_{fr}]\; X\) and \( Y_w \leq_{fr} \;[\geq_{fr}]\; Y\), then\: \(X_w \land Y_w \leq_{lr} \;[\geq_{lr}]\; (X \land Y)_w \)
\end{theorem}
\begin{proof}
Since \( S_X = S_Y \) and a common \( w(\cdot) \) is used to formulate \( X_w \) and \( Y_w \), we have \(S_1 = (l_1,\, u_1) = S_2 = (l_2,\, u_2)\). For \(x \in S_1 \cap S_2\), consider the PDFs
\[f_{X_w \land Y_w}(x)=f_{X_w}(x) \overline{F}_{Y_w}(x)+f_{Y_w}(x) \overline{F}_{X_w}(x), \]
and \[ f_{(X\land Y)_w}(x)= \frac{w^{\prime}(x)\overline{F}_{X}(x)\overline{F}_{Y}(x)}{\mathbb{E}[w(X\land Y)]}.\]
Now, for \(x \in S_1\cap S_2\), we have
 \begin{align}
 \frac{f_{X_w \land Y_w}(x) }{f_{(X\land Y)_w}(x)}&=\frac{f_{X_w}(x) \overline{F}_{Y_w}(x)+f_{Y_w}(x) \overline{F}_{X_w}(x)}{w^{\prime}(x)\overline{F}_{X}(x)\overline{F}_{Y}(x)}\times\mathbb{E}[w(X\land Y)] \nonumber\\
 &=\left[ \frac{w^{\prime}(x)\overline{F}_{X}(x)\overline{F}_{Y_w}(x)}{\mathbb{E}[w(X)]w^{\prime}(x)\overline{F}_{X}(x)\overline{F}_{Y}(x)}+\frac{w^{\prime}(x)\overline{F}_{Y}(x)\overline{F}_{X_w}(x)}{\mathbb{E}[w(Y)]w^{\prime}(x)\overline{F}_{X}(x)\overline{F}_{Y}(x)}\right]\times\mathbb{E}[w(X\land Y)] \nonumber\\
 &=\left[ \frac{\overline{F}_{Y_w}(x)}{\mathbb{E}[w(X)]\overline{F}_{Y}(x)}+\frac{\overline{F}_{X_w}(x)}{\mathbb{E}[w(Y)]\overline{F}_{X}(x)}\right]\times\mathbb{E}[w(X\land Y)]. \label{12}
 \end{align}
 Thus, using Equation~\eqref{12}, the assumptions of the theorem, and the fact that the sum of two non-negative decreasing [increasing] functions is again a decreasing [increasing] function, the theorem follows.  
 \end{proof}
 If we set \( w(x) = x \) in \cref{theorem7}, then the conditions \( \widetilde{X} \leq_{fr} X \) and \(\widetilde{Y} \leq_{fr} Y \) becomes equivalent to that \( X \) and \( Y \) are DMRL. We state the following corollary.
 \begin{corollary} Let \( X \) and \( Y \) be independent with \(S_X=S_Y\), and let \( \widetilde{X} \) and \( \widetilde{Y} \) also be independent. If \( X \) and \( Y \) are DMRL, then \( \widetilde{X} \land \widetilde{Y} \leq_{lr} (\widetilde{X \land Y}) \).
 \end{corollary}
\begin{remark} The above corollary was proved by \citet{bon2005ageing}.
\end{remark} \par
\cref{theorem7} can be generalized for more than two variables. We provide the following corollary.
\begin{corollary} Let \(X^1,\;X^2,\;\ldots,X^n\) be independent variables with common support, and let \(X^1_w,\;X^2_w,\;\ldots,X^n_w\) also be independent. If \(X^i_w \leq_{fr} \;[\geq_{fr}]\; X^i\)  for \(i=1,\;2,\,\cdots,\,n\), then \(X_w^1\land \cdots\land X_w^n \leq_{lr}\;[\geq_{lr}]\: (X^1\land \cdots\land X^n)_w\).
\end{corollary} 
\begin{remark} The proof of the above corollary is similar to the proof of \cref{theorem7}.
\end{remark} \par
The following theorem provides the conditions under which the usual stochastic ordering between \(X\) and \(Y\) remains preserved in their WTRVs \(X_w\) and \(Y_w\).
 \begin{theorem} \label{theorem8} Let \(X\) and \(Y\) be two independent variables. If \( w_1'(\cdot)/r_X(\cdot) \) is decreasing on \( \overline{S}_X = [l_X, u_X] \), \( w_2'(\cdot)/r_Y(\cdot) \) is increasing on \( \overline{S}_Y = [l_Y, u_Y] \), and \(
X \leq_{st} Y \), then \( \quad X_{w_1} \leq_{st} Y_{w_2}\).
\end{theorem}
\begin{proof} Since \( X \leq_{st} Y \), it follows that \( l_X \leq l_Y \) and \( u_X \leq u_Y \). Furthermore, by definition \( w_1(\cdot) \) and \( w_2(\cdot) \) have same supports as the variable \( X \) and \( Y \), i.e., \(S_X=(l_X,\,u_X)\) and \(S_Y=(l_Y,\,u_Y)\), respectively, it follows that \( l_1 = l_X \) and \( u_2 = u_Y \). Hence, we have \( l_X = l_1 \leq l_Y \leq l_2 \) and \( u_1 \leq u_X \leq u_Y = u_2 \). Thus, \( l_1 \leq l_2 \) and \( u_1 \leq u_2 \). If \( S_1 \cap S_2 = \emptyset \), i.e., \( l_1 = l_X \leq u_1 < l_2 \leq u_Y = u_2 \), then \( X_{w_1} \leq_{\text{st}} Y_{w_2} \) holds trivially. Now, if \( S_1 \cap S_2 \neq \emptyset \), i.e., \(l_1=l_X\leq l_2<u_1\leq u_Y=u_2\), then, \(S_1\cap S_2 =(l_2,\,u_1)\). To prove \(X_{w_1} \leq_{st} Y_{w_2}\), we fix \(\theta \in S_1\cap S_2\) and consider \[\Delta_3= \overline{F}_{Y_{w_2}}(\theta)-\overline{F}_{X_{w_1}}(\theta).\] 
It then follows that, \( l_1 \leq l_2 < \theta < u_1 \leq u_2 \), and the expression for \(\mathbb{E}[w_1(X)]\mathbb{E}[w_2(Y)]\Delta_3\) is
\begin{align}
\mathbb{E}[w_1(X)]\mathbb{E}[w_2(Y)]\Delta_3&=\left[ \int_{l_1}^{u_1}w_1'(x)\overline{F}_X(x)dx\right]\left[\int_{\theta}^{u_2}w_2'(y)\overline{F}_Y(y)dy\right]\nonumber\\
&\hspace{3cm}-\left[\int_{l_2}^{u_2}w_2'(y)\overline{F}_Y(y)dy\right]\left[\int_{\theta}^{u_1}w_1'(x)\overline{F}_X(x)dx\right]\nonumber\\
&=\left[ \int_{l_X}^{u_X}\frac{w_1'(x)}{r_X(x)}f_X(x)dx\right]\left[\int_{\theta}^{u_Y}\frac{w_2'(y)}{r_Y(y)}f_Y(y)dy\right]\nonumber\\
&\hspace{3cm}-\left[\int_{l_Y}^{u_Y}\frac{w_2'(y)}{r_Y(y)}f_Y(y)dy\right]\left[\int_{\theta}^{u_X}\frac{w_1'(x)}{r_X(x)}f_X(x)dx\right].\nonumber 
\end{align}
 We define the functions \[ \phi_2(x, y) = \frac{w_1'(x)w_2'(y)}{r_X(x)r_Y(y)} \cdot I(\theta \leq y \leq u_Y), \quad  \phi_1(x, y) = \frac{w_1'(x)w_2'(y)}{r_X(x)r_Y(y)} \cdot I(\theta \leq x \leq u_X), \] and \( \Delta\phi_{21}(x, y) = \phi_2(x, y) - \phi_1(x, y) \) for \( (x, y) \in \overline{S}_X \times \overline{S}_Y \). Under these definitions and the independence of \( X \) and \( Y \),  expression for \( \mathbb{E}[w_1(X)]\mathbb{E}[w_2(Y)] \Delta_3 \) becomes
\begin{align*}
\mathbb{E}[w_1(X)]\mathbb{E}[w_2(Y)]\Delta_3&= \int_{l_X}^{u_X}\int_{\theta}^{u_Y}\frac{w_1'(x)w_2'(y)}{r_X(x)r_Y(y)}f_X(x)f_Y(y))dydx\\
&\hspace{2cm}-\int_{l_Y}^{u_Y}\int_{\theta}^{u_X}\frac{w_1'(x)w_2'(y)}{r_X(x)r_Y(y)}f_Y(y)f_X(x)dxdy\\
&=\mathbb{E}[\phi_2(X,Y)]-\mathbb{E}[\phi_1(X,Y)],\\
\end{align*}
and for \((x,\,y)\in\overline{S}_X\times\overline{S}_Y\), we obtain
\[
\Delta\phi_{21}(x,\,y) = 
\begin{cases}
\dfrac{w_1'(x)w_2'(y)}{r_X(x)r_Y(y)}, & \text{if \:} l_X<x<\theta \text{\: and\: } \theta\leq y<u_Y\\[10pt]
-\dfrac{w_1'(x)w_2'(y)}{r_X(x)r_Y(y)}, & \text{if \:} \theta\leq x<u_X \text{ \:and\: } l_Y<y<\theta \\[4pt]
 0, & \text{otherwise.}
\end{cases}
\]
We now verify conditions (i)–(iv) of \cref{lemma1} for the independent variable \(X\) and \(Y\). Since \( S_X \cap S_Y = (l_Y, u_X) \neq \emptyset \), it follows that \( \Delta\phi_{21}(x, y) \geq 0 \) for every \( (x, y) \in \overline{S}_X \times \overline{S}_Y \) with \( x \leq y \). Using the non-negative and decreasing properties of \( w_1'(\cdot)/r_X(\cdot) \)  on \( \overline{S}_X \), and non-negative and increasing  properties of \( w_2'(\cdot)/r_Y(\cdot) \) on \( \overline{S}_Y \), we have for each fixed \( y \in \overline{S}_Y \), the \( \Delta\phi_{21}(x, y) \) is decreasing in \( x \) over the interval \( [l_X, \min(y, u_X)) \); similarly, for each fixed \( x \in \overline{S}_X \),  \( \Delta\phi_{21}(x, y) \) is increasing in \( y \) on the interval \( [\max(x, l_Y), u_Y) \). Furthermore, for all \( (x, y)\) in \( (\overline{S}_X \cap \overline{S}_Y) \times (\overline{S}_X \cap \overline{S}_Y) \) with \( x \leq y \), we obtain
\[
\Delta\phi_{21}(x, y) + \Delta\phi_{21}(y, x) = \frac{w_1'(x)w_2'(y)}{r_X(x)r_Y(y)} - \frac{w_1'(y)w_2'(x)}{r_X(y)r_Y(x)} \geq 0,
\]
using the monotonicity of \( w_1'(\cdot)/r_X(\cdot) \) and \( w_2'(\cdot)/r_Y(\cdot) \).
 Since \( X \leq_{st} Y \), the independent variable \(X\) and \(Y\) satisfy all the conditions of \cref{lemma1}, we have
\[
\mathbb{E}[w_1(X)]\mathbb{E}[w_2(Y)]\Delta_3 = \mathbb{E}(\phi_2(X,Y)) - \mathbb{E}(\phi_1(X,Y)) \geq 0,
\]
which further implies that \( \overline{F}_{Y_{w_2}}(\theta) \geq \overline{F}_{X_{w_1}}(\theta) \) on \(S_1\cap S_2\).
This completes the proof.  
\end{proof} \par
The following theorem provides conditions under which the failure rate order between \( X \) and \( Y \) is preserved by \( X_w \) and \( Y_w \).
\begin{theorem} \label{theorem9} Let \(X\) and \(Y\) be two independent variables. If \( l_1 \leq l_2 \), \( u_1 \leq u_2 \), \( w_1'(\cdot)/r_X(\cdot) \) is a decreasing function on \( S_X = (l_X,\, u_X) \), \( w_2'(\cdot)/r_Y(\cdot) \) is an increasing function on \( S_Y = (l_Y,\, u_Y) \), and \(X \leq_{fr} Y \), then
\(\: X_{w_1} \leq_{fr} Y_{w_2}\).
\end{theorem}
\begin{proof} From the ordering \(X \leq_{fr} Y\), we have \(l_X \leq l_Y\) and \(u_X \leq u_Y\). In the case of \(S_1 \cap S_2 = \emptyset\), i.e., when \(l_1 < u_1 \leq l_2 < u_2\), \(X_{w_1} \leq_{fr} Y_{w_2}\) holds trivially. Now consider the case where \(S_1 \cap S_2 \neq \emptyset\), that is, \(l_1 \leq l_2 < u_1 \leq u_2\), so that \(S_1 \cap S_2 = (l_2,\, u_1)\). Now, for \( x \in (-\infty,\, l_2] \), we have \( \overline{F}_{Y_{w_2}}(x) = 1 \), and the ratio \( \overline{F}_{Y_{w_2}}(x) / \overline{F}_{X_{w_1}}(x) \) is increasing. For \( x \in [u_1,\, u_2) \), we have \( \overline{F}_{X_{w_1}}(x) = 0 \), and the ratio is infinity, assuming that \( c/0 = \infty \) for \( c > 0 \). Therefore, the ratio \( \overline{F}_{Y_{w_2}}(x) / \overline{F}_{X_{w_1}}(x) \) is increasing on \( (-\infty,\, l_2] \cup [u_1,\, u_2) \). To establish that the ratio \(\overline{F}_{Y_{w_2}}(x)/\overline{F}_{X_{w_1}}(x)\) is increasing on \( S_1 \cap S_2= (l_2, u_1) \), take arbitrary \(s\) and \(t\) with \( l_2 < s < t < u_1 \), and consider the expression
\[
\Delta_4 = \overline{F}_{Y_{w_2}}(t)\, \overline{F}_{X_{w_1}}(s) - \overline{F}_{Y_{w_2}}(s)\, \overline{F}_{X_{w_1}}(t).
\]
Then, using the PDFs of \( X_{w_1} \) and \( Y_{w_2} \), we have the expression for \(\mathbb{E}[w_1(X)]\mathbb{E}[w_2(Y)]\Delta_4 \) as follows 
\begin{align*}
 \mathbb{E}[w_1(X)]\mathbb{E}[w_2(Y)]\,\Delta_4 &=\biggl[\int_{t}^{u_2} w_2'(y)\,\overline{F}_Y(y)\,\mathrm{d}y\biggr]
\biggl[\int_{s}^{u_1} w_1'(x)\,\overline{F}_X(x)\,\mathrm{d}x\biggr]\\
&\qquad
-\,\biggl[\int_{s}^{u_2} w_2'(y)\,\overline{F}_Y(y)\,\mathrm{d}y\biggr]
\biggl[\int_{t}^{u_1} w_1'(x)\,\overline{F}_X(x)\,\mathrm{d}x\biggr]\\
&=\,
\biggl[\int_{t}^{u_2} \frac{w_2'(y)}{r_Y(y)}\,f_Y(y)\,\mathrm{d}y\biggr]
\biggl[\int_{s}^{u_1} \frac{w_1'(x)}{r_X(x)}\,f_X(x)\,\mathrm{d}x\biggr]\\
&\qquad
-\,\biggl[\int_{s}^{u_2} \frac{w_2'(y)}{r_Y(y)}\,f_Y(y)\,\mathrm{d}y\biggr]
\biggl[\int_{t}^{u_1} \frac{w_1'(x)}{r_X(x)}\,f_X(x)\,\mathrm{d}x\biggr].
\end{align*}
%The sign of \( \Delta_4 \) is entirely determined by \( X \) and \( Y \). 
Now, define the functions 
\[
\phi_2(x, y) = \dfrac{w_1^{\prime}(x)w_2^{\prime}(y)}{r_X(x)r_Y(y)} \cdot I(s \leq x < u_X) I(t \leq y < u_Y), \quad (x, y) \in S_X \times S_Y,
\]
\[
\phi_1(x, y) = \dfrac{w_1^{\prime}(x)w_2^{\prime}(y)}{r_X(x)r_Y(y)}\cdot I(t \leq x < u_X) I(s \leq y < u_Y), \quad (x, y) \in S_X \times S_Y,
\]
and $\Delta \phi_{21}(x, y) = \phi_2(x, y) - \phi_1(x, y)$, $(x, y) \in S_X \times S_Y$.  Now, using the independence of \( X \) and \( Y \), we have the expression for \(\mathbb{E}[w_1(X)]\mathbb{E}[w_2(Y)] \Delta_4\) as follows
\begin{align*}
\mathbb{E}[w_1(X)]\mathbb{E}[w_2(Y)]\Delta_4
&=  \int_t^{u_Y}\int_s^{u_X}\dfrac{w_1^{\prime}(x)}{r_X(x)} \dfrac{w_2^{\prime}(y)}{r_Y(y)} f_X(x) f_Y(y) dxdy\\
&\qquad\qquad - \int_s^{u_Y}\int_t^{u_X} \dfrac{w_2^{\prime}(y)}{ r_Y(y)}  \dfrac{w_1^{\prime}(x)}{r_X(x)}f_X(x)f_Y(y) dxdy.\\
&=\mathbb{E}[\phi_2(X, Y)] - \mathbb{E}[\phi_1(X, Y)].
\end{align*}
For $(x, y) \in S_X \times S_Y$,
\[
\Delta \phi_{21}(x, y) =
\begin{cases}
\dfrac{w_1^{\prime}(x)w_2^{\prime}(y)}{r_X(x)r_Y(y)}, & \text{if } s \leq x < t \text{ and } t \leq y < u_Y, \\[10pt]
-\dfrac{w_1^{\prime}(x)w_2^{\prime}(y)}{r_X(x)r_Y(y)}, & \text{if } t \leq x < u_X \text{ and } s \leq y < t, \\
0, & \text{otherwise}.
\end{cases}
\]
Since \( S_1 \subseteq S_X \), \( S_2 \subseteq S_Y \), and \( S_1 \cap S_2 \neq \emptyset \), it follows that \( S_X \cap S_Y \neq \emptyset \), \( S_X \cap S_Y = (l_Y, u_X) \), and \( l_X \leq l_Y \leq l_2 < s < t < u_1 \leq u_X \leq u_Y\). \par  
We now verify conditions (i)–(iii) of \cref{lemma2} for the independent variables \(X\) and \(Y\). From the definition of \( \Delta \phi_{21}(\cdot, \cdot) \), it is evident that \( \Delta \phi_{21}(x, y) \geq 0 \) for every \( (x, y) \in S_X \times S_Y \) and \( x \leq y \). Moreover, for any fixed \( x \in S_X \cap S_Y \), the increasing property of the function \( w_2'(\cdot)/r_Y(\cdot) \) ensures the increasing property of \( \Delta \phi_{21}(x, y) \) for \( y \in (x, u_Y) \). Furthermore, using the decreasing property of \( w_1'(\cdot)/r_X(\cdot) \), the increasing property of \( w_2'(\cdot)/r_Y(\cdot) \), and that \( x \leq y \), we get
\[\Delta\phi_{21}(x,\,y)+\Delta\phi_{21}(y,\,x)=\dfrac{w_1'(x)w_2'(y)}{r_X(x)r_Y(y)}-\dfrac{w_1'(y)w_2'(x)}{r_X(y)r_Y(x)}\geq 0,\]
 whenever $(x, y) \in S_X \times S_Y$. Since \( X \leq_{fr} Y \) and \( X \) and \( Y \) are independent variables satisfying all conditions of \cref{lemma2}, we have
\begin{align*}
\mathbb{E}[w_1(X)]\mathbb{E}[w_2(Y)] \Delta_4&=\mathbb{E}[\phi_2(X, Y)] - \mathbb{E}[\phi_1(X, Y)] \geq 0,
\end{align*}
which further implies that \( \Delta_4=\overline{F}_{Y_{w_2}}(t)\overline{F}_{X_{w_1}}(s) - \overline{F}_{Y_{w_2}}(s)\overline{F}_{X_{w_1}}(t) \geq 0\). Therefore, the function \( \overline{F}_{Y_{w_2}}(\cdot)/\overline{F}_{X_{w_1}}(\cdot) \) is increasing on \( (-\infty,\, u_2) \), which implies that \( X_{w_1} \leq_{fr} Y_{w_2} \). 
\end{proof}
\begin{example} Let \( X \) follow the standard uniform distribution and \( Y \) follow the exponential distribution with rate \( \lambda = 1 \). Then, \(r_X(\cdot)\) and \(\)\(r_Y(\cdot)\) are defined as
\[
r_X(x) = 
\dfrac{1}{1 - x},\quad 0 < x < 1, \quad \text{and} \quad r_Y(x) = 1\quad \text{for } x > 0,
\]
respectively. For \( x < 0 \), \( r_X(x) = r_Y(x) = 0 \); for \( 0 < x < 1 \), \( r_X(x) \geq r_Y(x) \). Hence, \( X \leq_{fr} Y \). Let \(w_1(\cdot)\) and \(w_2(\cdot)\) are defined by
\[
w_1(x) = e^x - 1, \quad x > 0, \quad \text{and} \quad w_2(x) = x, \quad x > 0,
\]
respectively. Now observe that
\[
\frac{w_1^{\prime}(x)}{r_X(x)} = e^x(1 - x), \quad 0<x <1, \quad \text{and} \quad \frac{w_2^{\prime}(x)}{r_Y(x)} = 1, \quad x > 0,
\]
which are decreasing and increasing functions, respectively. Hence, applying \cref{theorem9}, we deduce that \(X_{w_1} \leq_{fr} Y_{w_2}\).
\end{example}
\begin{remark}
In Example~7, \( w_2'(x)/w_1'(x) = e^{-x} \) is not increasing for \( x > 0 \), so the condition of \cref{theorem5} is not satisfied.
The PDFs of \(X_{w_1}\) and \(Y_{w_2}\) are given by
\[
f_{X_{w_1}}(x) = \frac{e^x(1 - x)}{e - 2}, \quad 0 < x < 1, \quad \text{and} \quad f_{Y_{w_2}}(x) = e^{-x},\quad x>0.
\]
Now, consider the function
\[
\frac{f_{Y_{w_2}}(x)}{f_{X_{w_1}}(x)} = \frac{e^{-2x}(e - 2)}{1 - x}, \quad 0<x<1.
\]
We plot \( f_{Y_{w_2}}(\cdot)/f_{X_{w_1}}(\cdot) \) to demonstrate that it is not increasing. Therefore, the likelihood ratio order \( X_{w_1} \leq_{lr} Y_{w_2} \) does not hold.
\begin{figure}[H]
  \centering
  \includegraphics[width=0.6\textwidth]{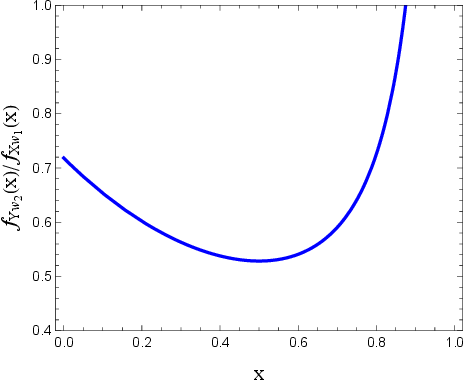} 
  \caption{\normalsize Plot of the function \(f_{Y_{w_2}}(\cdot)/f_{X_{w_1}}(\cdot)\).}
  \label{figexample7}
\end{figure}
\end{remark} \par
The following theorem provides conditions under which the reversed failure rate order between \( X \) and \( Y \) is preserved by \( X_w \) and \( Y_w \).
\begin{theorem} \label{theorem10} Let \(X\) and \(Y\) be two independent variables. If \( l_1 \leq l_2 \), \( u_1 \leq u_2 \), \( w_1'(\cdot)/r_X(\cdot) \) is a decreasing function on \( S_X = (l_X,\, u_X) \), \( w_2'(\cdot)/r_Y(\cdot) \) is an increasing function on \( S_Y = (l_Y,\, u_Y) \), and \(X \leq_{rfr} Y \), then
\(\: X_{w_1} \leq_{rfr} Y_{w_2}\).
\end{theorem}
\begin{proof} From the ordering \(X \leq_{rfr} Y\), we have \(l_X \leq l_Y\) and \(u_X \leq u_Y\). In the case of \(S_1 \cap S_2 = \emptyset\), i.e., when \(l_1 < u_1 \leq l_2 < u_2\), \(X_{w_1} \leq_{rfr} Y_{w_2}\) holds trivially. Now consider the case where \(S_1 \cap S_2 \neq \emptyset\), that is, \(l_1 \leq l_2 < u_1 \leq u_2\), so that \(S_1 \cap S_2 = (l_2,\, u_1)\). We now show that \( F_{Y_{w_2}}(x) /F_{X_{w_1}}(x) \) is increasing on \((l_1,\:\infty)\). For \( x \in (l_1,\, l_2] \), we have \( F_{Y_{w_2}}(x) = 0 \), and consequently the ratio \( F_{Y_{w_2}}(x) / F_{X_{w_1}}(x) \) is zero. For \( x \in [u_1,\, \infty) \), we have \( F_{X_{w_1}}(x) = 1 \), and hence  \( F_{Y_{w_2}}(x) / F_{X_{w_1}}(x)=F_{Y_{w_2}}(x) \) , which is again an increasing function. Therefore, the ratio \( F_{Y_{w_2}}(x) / F_{X_{w_1}}(x) \) is increasing on \( (l_1,\, l_2] \cup [u_1,\, \infty) \). To establish that the ratio \(F_{Y_{w_2}}(x)/F_{X_{w_1}}(x)\) is increasing on \( S_1 \cap S_2= (l_2, u_1) \), take arbitrary \(s\) and \(t\) with \( l_2 < s < t < u_1 \), and consider the expression
\[
\Delta_5 = F_{Y_{w_2}}(t)\, F_{X_{w_1}}(s) - F_{Y_{w_2}}(s)\, F_{X_{w_1}}(t).
\]
Then, using the PDFs of \( X_{w_1} \) and \( Y_{w_2} \), we have the expression for \(\mathbb{E}[w_1(X)]\mathbb{E}[w_2(Y)]\Delta_5 \) as follows 
\begin{align*}
 \mathbb{E}[w_1(X)]\mathbb{E}[w_2(Y)]\,\Delta_5 &=\biggl[\int_{l_2}^{t} w_2'(y)\,\overline{F}_Y(y)\,\mathrm{d}y\biggr]
\biggl[\int_{l_1}^{s} w_1'(x)\,\overline{F}_X(x)\,\mathrm{d}x\biggr]\\
&\qquad
-\,\biggl[\int_{l_2}^{s} w_2'(y)\,\overline{F}_Y(y)\,\mathrm{d}y\biggr]
\biggl[\int_{l_1}^{t} w_1'(x)\,\overline{F}_X(x)\,\mathrm{d}x\biggr]\\
&=\,
\biggl[\int_{l_2}^{t} \frac{w_2'(y)}{r_Y(y)}\,f_Y(y)\,\mathrm{d}y\biggr]
\biggl[\int_{l_1}^{s} \frac{w_1'(x)}{r_X(x)}\,f_X(x)\,\mathrm{d}x\biggr]\\
&\qquad
-\,\biggl[\int_{l_2}^{s} \frac{w_2'(y)}{r_Y(y)}\,f_Y(y)\,\mathrm{d}y\biggr]
\biggl[\int_{l_1}^{t} \frac{w_1'(x)}{r_X(x)}\,f_X(x)\,\mathrm{d}x\biggr].
\end{align*}
%The sign of \( \Delta_4 \) is entirely determined by \( X \) and \( Y \). 
Now, define the functions 
\[
\phi_2(x, y) = \dfrac{w_1^{\prime}(x)w_2^{\prime}(y)}{r_X(x)r_Y(y)} \cdot I(l_X \leq x < s) I(l_Y \leq y < t), \quad (x, y) \in S_X \times S_Y,
\]
\[
\phi_1(x, y) = \dfrac{w_1^{\prime}(x)w_2^{\prime}(y)}{r_X(x)r_Y(y)}\cdot I(l_X \leq x < t) I(l_Y\leq y < s), \quad (x, y) \in S_X \times S_Y,
\]
and $\Delta \phi_{21}(x, y) = \phi_2(x, y) - \phi_1(x, y)$, $(x, y) \in S_X \times S_Y$.  Now, using the independence of \( X \) and \( Y \), we have the expression for \(\mathbb{E}[w_1(X)]\mathbb{E}[w_2(Y)] \Delta_5\) as follows
\begin{align*}
\mathbb{E}[w_1(X)]\mathbb{E}[w_2(Y)]\Delta_5
&=  \int_{l_Y}^{t}\int_{l_X}^s\dfrac{w_1^{\prime}(x)}{r_X(x)} \dfrac{w_2^{\prime}(y)}{r_Y(y)} f_X(x) f_Y(y) dxdy\\
&\qquad\qquad - \int_{l_Y}^s\int_{l_X}^t \dfrac{w_2^{\prime}(y)}{ r_Y(y)}  \dfrac{w_1^{\prime}(x)}{r_X(x)}f_X(x)f_Y(y) dxdy.\\
&=\mathbb{E}[\phi_2(X, Y)] - \mathbb{E}[\phi_1(X, Y)].
\end{align*}
For $(x, y) \in S_X \times S_Y$,
\[
\Delta \phi_{21}(x, y) =
\begin{cases}
\dfrac{w_1^{\prime}(x)w_2^{\prime}(y)}{r_X(x)r_Y(y)}, & \text{if } l_X \leq x < s \text{ and } s \leq y < t, \\[10pt]
-\dfrac{w_1^{\prime}(x)w_2^{\prime}(y)}{r_X(x)r_Y(y)}, & \text{if } s \leq x < t \text{ and } l_Y \leq y < s, \\
0, & \text{otherwise}.
\end{cases}
\] 
Since \( S_1 \subseteq S_X \), \( S_2 \subseteq S_Y \), and \( S_1 \cap S_2 \neq \emptyset \), it follows that \( S_X \cap S_Y \neq \emptyset \), \( S_X \cap S_Y = (l_Y, u_X) \), and \( l_X \leq l_Y \leq l_2 < s < t < u_1 \leq u_X \leq u_Y\). \par  
We now verify conditions (i)–(iii) of \cref{lemma3} for the independent variables \(X\) and \(Y\). From the definition of \( \Delta \phi_{21}(\cdot, \cdot) \), it is evident that \( \Delta \phi_{21}(x, y) \geq 0 \) for every \( (x, y) \in S_X \times S_Y \) and \( x \leq y \). Moreover, for any fixed \( y \in S_X \cap S_Y \), the decreasing property of the function \( w_1'(\cdot)/r_X(\cdot) \) ensures the decreasing property of \( \Delta \phi_{21}(x, y) \) for \( x \in (l_X,\, y) \). Furthermore, using the decreasing property of \( w_1'(\cdot)/r_X(\cdot) \), the increasing property of \( w_2'(\cdot)/r_Y(\cdot) \), and that \( x \leq y \), we get
\[\Delta\phi_{21}(x,\,y)+\Delta\phi_{21}(y,\,x)=\dfrac{w_1'(x)w_2'(y)}{r_X(x)r_Y(y)}-\dfrac{w_1'(y)w_2'(x)}{r_X(y)r_Y(x)}\geq 0,\]
 whenever $(x, y) \in S_X \times S_Y$. Since \( X \leq_{rfr} Y \), and \( X \) and \( Y \) are independent variables satisfying all conditions of \cref{lemma3}, it follows from \cref{lemma3} that
\begin{align*}
\mathbb{E}[w_1(X)]\mathbb{E}[w_2(Y)] \Delta_5&=\mathbb{E}[\phi_2(X, Y)] - \mathbb{E}[\phi_1(X, Y)] \geq 0,
\end{align*}
which further implies that \( \Delta_5=F_{Y_{w_2}}(t) F_{X_{w_1}}(s) - F_{Y_{w_2}}(s)F_{X_{w_1}}(t) \geq 0\). Therefore, the function \( F_{Y_{w_2}}(\cdot)/F_{X_{w_1}}(\cdot) \) is increasing on \( (l_1,\, \infty) \), which implies that \( X_{w_1} \leq_{rfr} Y_{w_2} \). 
\end{proof} 
\begin{example}
Let \( X \) follow the exponential distribution with rate \( \lambda = 1 \), and let \( Y \) follow the standard uniform distribution. Then, we have
\[ \frac{F_Y(x)}{F_X(x)}=\frac{x}{1-e^{-x}}\]
is increasing on \((0,\,\infty)\). Hence, \( X \leq_{rfr} Y \). Now consider \( w_1(x) = x,\; x > 0 \) and \( w_2(x) = -x - \ln(1 - x),\; 0 < x < 1 \). Then we have
\[\frac{w_1'(x)}{r_X(x)} = 1,\quad x > 0, \qquad \text{and} \qquad \frac{w_2'(x)}{r_Y(x)} = \frac{x(1 - x)}{1 - x} = x,\quad 0 < x < 1,\]
which are decreasing and increasing functions, respectively. Therefore, by applying \cref{theorem10}, it follows that \( X_{w_1} \leq_{rfr} Y_{w_2} \).
\end{example}
\section{An Application of the Framework}
In this section, we leverage the proposed framework to derive a continuous and bounded random variable \( X_w \), under the assumption that \( X \) follows a Kumaraswamy distribution. Subsequently, we model real-world data using the random variables \( X \), \( X_w \), and the Beta distribution, and assess their goodness-of-fit metrics by employing the method of maximum likelihood estimation (MLE), followed by statistical tests to compare their fitting performance. Poondi Kumaraswamy introduced a family of double-bounded continuous probability distributions in \citet{KUMARASWAMY198079} to model data with finite lower and upper bounds, incorporating zero inflation, that is, the left end point is zero and lies within the support of the distribution. The Kumaraswamy distribution is comparable to the beta distribution and provides enhanced simplicity, especially for simulation-based applications. This advantage of Kumaraswamy distribution comes from the closed-form availability of its PDF, CDF, and quantile function.\par
We model the data using \( X \), defined on the interval \( (0,1) \). To construct the WTRV \( X_w \), we use the weight function \( w(x) = x^c \), \( x \in (0,1) \), where \( c > 0 \). The resulting distribution of \( X_w \) is called the weighted Kumaraswamy (or simply, weighted Kw) distribution. We demonstrate that \(X_w\) provides a substantially better goodness-of-fit measures compared to \(X\) when modeling rainfall datasets\footnote{\scriptsize{\url{https://www.data.gov.in/resource/rainfall-nw-india-and-its-departure-normal-monsoon-session-1901-2021}}}$^{,\,}$\footnote{\scriptsize{\url{https://www.data.gov.in/resource/rainfall-ne-india-and-its-departure-normal-monsoon-session-1901-2021}}} from Northwest and Northeast India for the period 1901--2021. We denote \( X \sim \text{Kw}(a, b) \) \(\left[ X \sim \text{WK}(a, b, c) \right] \) to say that \( X \) follows the Kumaraswamy \(\left[\text{weighted Kumaraswamy}\right]\) distribution with shape parameters \( a, b > 0 \) \(\left[ a, b, c > 0 \right] \).
\subsection{\textbf{Model Description}} 
The PDF of \( X \sim \text{Kw}(a,\,b) \), as given in \citet{Lemonte01122011}, is  
\[f_X(x)=abx^{a-1}(1-x^a)^{b-1},\quad x\in (0,\,1),\quad
a,\,b>0,\]
 the SF \(\overline{F}_X(\cdot)\) and quantile function \(F_X^{-1}(\cdot)\)  of \(X\) are \[\overline{F}_X(x)=(1-x^a)^b\quad \text{and}\quad F_X^{-1}(x)=(1-(1-x)^{\frac{1}{b}})^{\frac{1}{a}},\quad x\in(0,\,1),\]
 respectively, and its the \(n\)-th raw moment is \[\mathbb{E}[X^n]=ab\int_{0}^{1}x^{a+n-1}(1-x^a)^{b-1}dx=bB\big(1+\frac{n}{a},\, b\big),\quad n\in\mathbb{N},\] where \[B(p,q)=\int_0^1 t^{p-1} (1 - t)^{q-1} dt,\quad p,\,q>0.\]
The PDF of \( X_w \sim \text{WK}(a, b, c) \) (see \cref{table1}), is  
\begin{equation} f_{X_w}(x)=\frac{cx^{c-1}(1-x^a)^b}{bB\big(1+\frac{c}{a},\, b\big)},\quad x\in (0,\,1),\quad
a,\,b,\,c>0,\label{13}
\end{equation}
the SF of \(X_w\) is \begin{align} 
\overline{F}_{X_w}(x)&=\frac{c}{bB\big(1+\frac{c}{a},\, b\big)}\int_{x}^{1}t^{c-1}(1-t^a)^b dt\nonumber \\
&=\frac{c}{abB\big(1+\frac{c}{a},\, b\big)}\int_{x^a}^{1}z^{\frac{c}{a}-1}(1-z)^bdz \nonumber\\
&=\frac{cB_1\big(x^a,1; \frac{c}{a}, b+1\big)}{abB\big(1+\frac{c}{a},\, b\big)}, \nonumber
\end{align}
where \[B_1(y,1; p, q) = \int_y^1 t^{p-1} (1 - t)^{q-1} dt,\quad 0 \leq y \leq 1,\quad p, q > 0,\] and the \( n \)-th raw moment of \( X_w \) is
 \[\mathbb{E}[X_w^n]=\int_{0}^{1}\frac{cx^{c+n-1}(1-x^a)^b }{bB\big(1+\frac{c}{a}, b\big)}dx=\frac{cB\big(\frac{c+n}{a}, b+1\big)}{abB\big(1+\frac{c}{a}, b\big)},\quad n \in \mathbb{N}.
\]
\begin{figure}[H]
 \centering
 {\includegraphics[height=7 cm, width=11 cm]{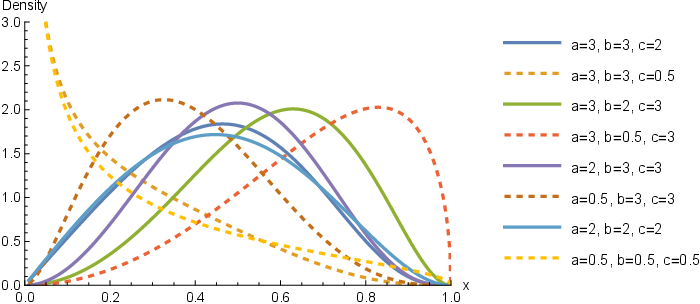}} 
 \caption{Density plot of weighted Kumaraswamy distribution with parameters \(a,\,b,\,c\).} 
\end{figure}
\subsection{\textbf{Estimation of Model Parameters}}
The weighted Kumaraswamy distribution shares similarities with both the beta and Kumaraswamy distributions, since all three are continuous random variables defined on the interval \( ( 0, \, 1) \). To assess their comparative modeling performance, we model monsoon rainfall datasets to these distributions using the method of MLE. Suppose we are given the data \(\underline{z}=\{z_1,\,z_2,\cdots,z_n\}\) such that \(z_{\text{min}}\), \(z_{\text{max}}\) are the minimum and maximum observations in \(\underline{z}\), respectively. To fit the dataset to the considered models, we normalize the observations as follows: \[ x_i = \frac{z_i - z_{\text{min}}}{z_{\text{max}} - z_{\text{min}}}, \quad i \in \{1, 2, \dots, n\}. \] The log-likelihood function \(L\) obtained using Equation~\eqref{13} as follows:
\begin{align}
 L(a,b,c)&=n\log\left(\frac{c}{bB(1+c/a, b)}\right)+(c-1)\sum_{i=1}^{n}\log(x_i)+b\sum_{i=1}^{n}\log(1-x_i^{a}). \label{14}
 \end{align}
We obtain the first-order partial derivatives of \( L(a, b, c) \) with respect to \( a, b, \) and \( c \) and equate them to zero to get the likelihood equations. However, the likelihood equations cannot be solved analytically because of the presence of \( a \), \( b \), and \( c \) as arguments of the beta function in the log-likelihood expression. Thus, we employ the "Limited-memory Broyden–Fletcher–Goldfarb– Shanno with Box" (L-BFGS-B) constrained optimisation method and the \textit{mle()} function in \textsf{R} software to estimate the model parameters, and effectively minimize \( -L(a, b, c)\) to obtain the maximum likelihood estimates of \( a, b, \) and \( c \). Similarly, we estimate the model parameters for the beta and Kumaraswamy distributions employing the same optimisation technique. \cref{table2} presents the descriptive statistics of monsoon rainfall data (in millimeters) from Northwest (NW) and Northeast (NE) India, while \cref{table3} shows the estimated parameters of the fitted probability models for these datasets.
\begin{table}[H]
\caption{{\large Descriptive statistics of the datasets}}
\label{table2} 
\centering
{\footnotesize \begin{tabular}{cccccccccc}
\hline
&&&&&&&&&\\
Dataset & Mean & Median & Mode & Std. Dev. & Variance & Skewness & Kurtosis & Min & Max \\
\hline &&&&&&&&&\\
  NW India 
& 593.595 & 602.8 & 557.2 & 104.84 & 10992.46 & -0.056 & 3.07 & 338.8 & 928.4 \\ &&&&&&&&&\\
 NE India 
& 1491.061 & 1486.8 & 1783 & 200.48 & 40192.68 & 0.226 & 2.83 & 1078.9 & 2090.3 \\ &&&&&&&&&\\
\hline
\end{tabular}}
\end{table}
\begin{table}[H]
    \centering
    \caption{{\large Estimated model parameters using the method of MLE }}
    \label{table3} 
   \begin{tabular}{lcccc}
    \toprule
    \multirow{2}{*}{Dataset} & \multirow{2}{*}{Probability Model} & \multicolumn{3}{c}{Estimated Parameters} \\
    \cmidrule(lr){3-5}
    & & \(a\) & \(b\) & \(c\) \\
    \midrule
        \multirow{3}{*}{\begin{minipage}{2cm} 
               
                 NW India
            \end{minipage}} 
        & Beta            & 3.1297  & 4.2115  & --      \\
        & Kumaraswamy (Kw)  & 2.5667  & 5.8036  & --      \\
        & Weighted Kw     & 2.0092  & 13.2024 & 6.1306  \\
        &&&&\\
        \multirow{3}{*}{\begin{minipage}{2cm} 
             
                  NE India
            \end{minipage}} 
        & Beta            & 2.0394  & 3.0538  & --      \\
        & Kw  & 1.8521  & 3.5741  & --      \\
        & Weighted Kw     & 1.6085  & 3.6069  & 3.2833  \\
        \bottomrule
    \end{tabular}
\end{table}
\begin{figure}[H]
 \centering
 \begin{minipage}{0.5\textwidth}
  \centering
  \includegraphics[height=6cm, width=7.5cm]{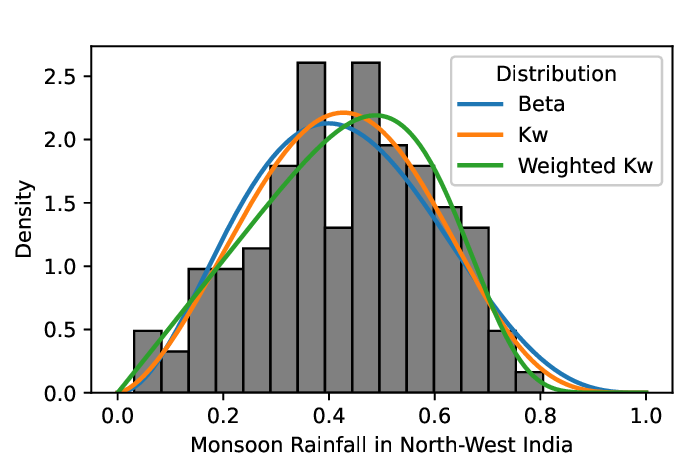}
 \end{minipage}%
 \begin{minipage}{0.5\textwidth}
  \centering
  \includegraphics[height=6cm, width=7.5cm]{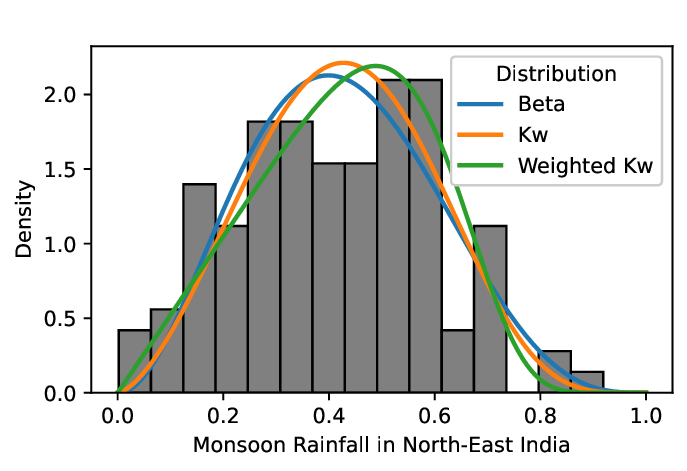}
 \end{minipage}
 \caption{Plots of the histograms and corresponding fitted densities of the datasets.}
\end{figure} 
\begin{table}[H]
    \centering
    \caption{{\large Goodness-of-fit and Model Selection Criteria}} \label{table4} 
   \begin{tabular}{llccc}
    \toprule
    \multirow{2}{*}{Dataset} & \multirow{2}{*}{Metric} & \multicolumn{3}{c}{Probability Model} \\
    \cmidrule(lr){3-5}
    & & Beta & Kw & Weighted Kw \\
    \midrule
        \multirow{4}{*}{\begin{minipage}{2cm} 
                \centering
                  NW India
            \end{minipage}} 
        & Log-likelihood  & 43.3819  & 45.6155  & 47.6259  \\
        & AIC            & -82.7638 & -87.2309 & -89.2517 \\
        & BIC            & -77.2055 & -81.6727 & -80.9144 \\
        & RMSE           & 0.2057   & 0.1448   & 0.1052   \\
     &&&&\\
        \multirow{4}{*}{\begin{minipage}{2cm} 
                \centering
                  NE India
            \end{minipage}} 
        & Log-likelihood  & 28.1876  & 29.3296  & 30.0229  \\
        & AIC            & -52.3752 & -54.6593 & -54.0457\\
        & BIC            & -46.8509 & -49.1349 & -45.7592 \\
        & RMSE           & 0.1753   & 0.1377   & 0.0958   \\
        \bottomrule
    \end{tabular}
\end{table}
\subsection{\textbf{ Statistical Testing for the Goodness-of-Fit of the Datasets}}
We perform several goodness-of-fit tests based on empirical CDFs for the selection of the most suitable probability models for the datasets. We use the Kolmogorov-Smirnov (KS), Anderson-Darling (AD), Cramér-von Mises (CvM), and Chi-Square (with 10 bins) (ChiSq) goodness-of-fit tests to determine the goodness of the fitted models. Let \( F(\cdot) \) represent the distribution function of a probability model obtained by estimating the model parameters for a given dataset, and let \( F_n(\cdot) \) represent the empirical CDF of that dataset. Now, the underlying hypothesis for a goodness-of-fit test is given by  
\[
H_0: F_n(x) = F(x), \quad \forall\, x \quad \text{vs.} \quad H_1: F_n(x) \neq F(x), \quad \text{for at least one } x.
\]  
The above-mentioned hypothesis tests allow us to statistically determine whether a fitted model significantly deviates from the empirical distribution of the given dataset. \par
\cref{table5} provides the results of the goodness-of-fit tests for the three candidate models across each considered dataset.
\begin{table}[H]
    \centering
    \renewcommand{\arraystretch}{1.75}
    \caption{{\large Goodness-of-fit test statistics and p-value for the considered models}} \label{table5} 
   {\footnotesize \begin{tabular}{c c ccc ccc}
        \toprule
        \multirow{2}{*}{Dataset} & \multirow{2}{*}{Test} & \multicolumn{3}{c}{Test statistic values} & \multicolumn{3}{c}{p-value} \\
        \cmidrule(lr){3-5} \cmidrule(lr){6-8}
        & & Beta & Kw & Weighted Kw  & Beta & Kw & Weighted Kw  \\
        \midrule
        \multirow{4}{*}{\begin{minipage}{2cm} 
                
                  NW India
            \end{minipage}} 
        & KS      & 0.0665  & 0.0513  & 0.0451 & 0.6677  & 0.9129  & 0.9687  \\
        & AD      & 0.8442  & 0.4813  & 0.1950  & 0.4500  & 0.7656  & 0.9917  \\
        & CvM     & 0.1247  & 0.0627  & 0.0260  & 0.4772  & 0.7978  & 0.9877  \\
        & ChiSq   & 9.2350 & 6.6068 & 5.8144 & 0.1608  & 0.3587  & 0.3247  \\
        \midrule
        \multirow{4}{*}{\begin{minipage}{2cm} 
                
                  NE India
            \end{minipage}} 
        & KS      & 0.0570  & 0.0543  & 0.0484  & 0.8417  & 0.8800  & 0.9468  \\
        & AD      & 0.6728  & 0.4713  & 0.4024  & 0.5816  & 0.7759  & 0.8461  \\
        & CvM     & 0.0885  & 0.0569  & 0.0422  & 0.6453  & 0.8340  & 0.9223  \\
        & ChiSq  & 4.3261 & 3.4122 & 3.5281 & 0.7415  &0.8444  & 0.7402  \\
        \bottomrule
    \end{tabular}}
\end{table} 
\section{Conclusion}
This study introduces a novel framework for formulating a non-negative continuous probability distribution by utilizing the SF of a non-negative continuous random variable and a weight function. The proposed framework has the potential to construct probability distributions that can model real-world data with superior goodness-of-fit metrics than the input random variable of the framework, as evidenced by the weighted Kumaraswamy distribution, which exhibits a superior goodness-of-fit compared to the original Kumaraswamy probability model. Future research may comprehensively investigate the proposed framework, with particular emphasis on the tail behavior of the WTRV of a random variable. Special focus may also be given to its ability to generate flexible probability distributions useful in various domains, including reliability theory and survival analysis.
 \section*{Acknowledgements} The first author gratefully acknowledges the research fellowship granted by the Ministry of Education, Government of India.
\bibliographystyle{custom-harv}
\bibliography{references}
\end{document}